\documentclass[1 [leqno,11pt]{amsart}
\usepackage{amssymb, amsmath,amsmath,latexsym,amssymb,amsfonts,amsbsy, amsthm}
\def\normsup#1{\|#1\|_{L^\infty}}
\def\normld#1{\|#1\|_{L^2}}

\def\inte#1{
\displaystyle\mathop{#1\kern0pt}^\circ }

\newcommand{\ef}{ \hfill $ \blacksquare $ \vskip 3mm}
\newcommand{\beqo}{\begin{equation*}}
\newcommand{\eeqo}{\end{equation*}}
\newcommand{\beno}{\begin{eqnarray*}}
\newcommand{\eeno}{\end{eqnarray*}}
\def\normld#1{\|#1\|_{L^2}}
\def\normsup#1{\|#1\|_{L^\infty}}
\def\normp#1{\|#1\|_{L^p}}
\def\normop#1{\|#1\|_{L^1_t(L^p)}}
\def\normo#1{\|#1\|_{L^1_t(L^2)}}

\def\normbp#1#2{\|#1\|_{B^{#2}_{p,1}}}
\def\normbq#1#2{\|#1\|_{B^{#2}_{q,1}}}
\def\normb#1#2{\|#1\|_{B^{#2}_{2,1}}}
\def\normBp#1#2{\|#1\|_{\widetilde{L}_t^{\infty}(B^{#2}_{p,1})}}
\def\normBq#1#2{\|#1\|_{\widetilde{L}_t^{\infty}(B^{#2}_{q,1})}}
\def\normB#1#2{\|#1\|_{\widetilde{L}_t^{\infty}(B^{#2}_{2,1})}}
\def\normBpo#1#2{\|#1\|_{L_t^1(B^{#2}_{p,1})}}

\def\normBo#1#2{\|#1\|_{L_t^1(B^{#2}_{2,1})}}

\numberwithin{equation}{section}

\let\pa=\partial
\let\al=\alpha

\let\e=\varepsilon

\let\lam=\lambda
\let\r=\rho
\let\s=\sigma
\let\f=\frac

\let\p=\psi

\let\D=\Delta

\let\Om=\Omega
\let\wt=\widetilde

\let\tri=\Delta
\let\ep=\epsilon

\def\cB{{\mathcal B}}
\def\cC{{\mathcal C}}

\def\cF{{\mathcal F}}

\def\cM{{\mathcal M}}

\def\cS{{\mathcal S}}

\def\virgp{\raise 2pt\hbox{,}}
\def\cdotpv{\raise 2pt\hbox{;}}
\def\eqdef{\buildrel\hbox{\footnotesize d\'ef}\over =}
\def\eqdefa{\buildrel\hbox{\footnotesize def}\over =}

\def\C{\mathop{\mathbb C\kern 0pt}\nolimits}
\def\DD{\mathop{\mathbb D\kern 0pt}\nolimits}
\def\EE{\mathop{{\mathbb E \kern 0pt}}\nolimits}
\def\K{\mathop{\mathbb K\kern 0pt}\nolimits}
\def\N{\mathop{\mathbb N\kern 0pt}\nolimits}
\def\Q{\mathop{\mathbb Q\kern 0pt}\nolimits}
\def\R{\mathop{\mathbb R\kern 0pt}\nolimits}
\def\SS{\mathop{\mathbb S\kern 0pt}\nolimits}
\def\ZZ{\mathop{\mathbb Z\kern 0pt}\nolimits}
\def\TT{\mathop{\mathbb T\kern 0pt}\nolimits}
\def\P{\mathop{\mathbb P\kern 0pt}\nolimits}

\newcommand{\la}{\lambda}

\newcommand{\Z}{{\ZZ}}

\def\dv{\mbox{div}}

\def\dive{\mathop{\rm div}\nolimits}

\def\Supp{\mathop{\rm Supp}\nolimits\ }

\def\no{\noindent}
\def\na{\nabla}
\def\p{\partial}

\newcommand{\beq}{\begin{equation}}
\newcommand{\eeq}{\end{equation}}
\newcommand{\ben}{\begin{eqnarray}}
\newcommand{\een}{\end{eqnarray}}
\newcommand{\andf}{\quad\hbox{and}\quad}

\newtheorem{defi}{Definition}[section]
\newtheorem{thm}{Theorem}[section]
\newtheorem{lem}{Lemma}[section]
\newtheorem{rmk}{Remark}[section]
\newtheorem{col}{Corollary}[section]
\newtheorem{prop}{Proposition}[section]
\renewcommand{\theequation}{\thesection.\arabic{equation}}
\begin{document}
\title[Global solutions of $2-$D inhomogeneous NS equations]
{ Global solutions  to 2-D inhomogeneous
 Navier-Stokes system with general velocity }
\author[J. Huang]{Jingchi Huang}\address[J. HUANG]
{Academy of Mathematics $\&$ Systems Science, Chinese Academy of
Sciences, Beijing 100190, P. R. CHINA} \email{jchuang@amss.ac.cn}
\author[M. PAICU]{Marius Paicu}
\address [M. PAICU]
{Universit\'e  Bordeaux 1\\
 Institut de Math\'ematiques de Bordeaux\\
F-33405 Talence Cedex, France}
\email{marius.paicu@math.u-bordeaux1.fr}
\author[P. ZHANG]{Ping Zhang}%
\address[P. ZHANG]
 {Academy of
Mathematics $\&$ Systems Science and  Hua Loo-Keng Key Laboratory of
Mathematics, The Chinese Academy of Sciences\\
Beijing 100190, CHINA } \email{zp@amss.ac.cn}
\date{9/9/2012}
\maketitle
 \begin{abstract} In this paper, we are concerned with the global
wellposedness of 2-D  density-dependent incompressible Navier-Stokes
equations \eqref{INS} with variable viscosity, in a critical
functional framework which is invariant by the  scaling of the
equations and under a non-linear smallness condition on fluctuation
of the initial density which has to be doubly exponential small
compared with the size of the initial velocity. In the second part
of the paper, we apply our methods combined with the techniques in
\cite{dm} to prove the global existence of solutions to \eqref{INS}
 with piecewise constant initial density which has small jump at
the interface and is away from vacuum. In particular, this latter
result removes the smallness condition for the initial velocity  in
a corresponding theorem of \cite{dm}.

\end{abstract}


\vskip 0.2cm


\renewcommand{\theequation}{\thesection.\arabic{equation}}
\setcounter{equation}{0}
\section{Introduction}
In this paper, we consider the global existence of solutions to the
following 2-D incompressible inhomogeneous  Navier-Stokes equations
with initial data in the scaling invariant Besov spaces and without
size restriction for the  initial velocity:
\begin{equation}
\label{INS}
 \left\{\begin{array}{l}
\displaystyle \pa_t \rho + u \cdot \na \rho=0,\qquad (t,x)\in\R^+\times\R^2,\\
\displaystyle \pa_t(\rho u) + \dive(\rho u \otimes u) -\mathrm{div}( \mu(\rho)\cM) +\na \Pi=0, \\
\displaystyle \mathrm{div} u = 0,
\end{array}\right.
\end{equation}
where $\rho, u=(u_1,u_2)$ stand for the density and  velocity of the
fluid respectively, $\mathcal{M} = \frac{1}{2}(\pa_{i}u_{j}+\pa_{j}
u_{i}),$ $\Pi$  is a scalar pressure function,
 and the
viscosity coefficient $\mu(\rho)$ is a smooth, positive function on
$[0,\infty).$ Such system describes a fluid which is obtained by
mixing two immiscible fluids that are incompressible and that have
different densities. It may also describe a fluid containing a
melted substance.

 There is a wide literatures  devoted to the mathematical study of
the incompressible Navier-Stokes equations in the homogeneous case
(where the density is a constant) or in the more physical case of
inhomogeneous fluids. In the homogeneous case, the celebrated
theorem of  Leray \cite{Leray} on  the existence of global weak
solutions with finite energy in any space dimension is now a
classical result. Moreover,  in the two dimensional space, it is
also classical that the Leray weak solution is in fact a global
strong solution. In dimension larger than two, the Fujita-Kato
theorem \cite{fk} allows to construct global strong solutions under
a smallness condition on the initial data comparing with the
viscosity of the fluid. To obtain those types of results in the
inhomogeneous case are the topics of many recent  works dedicated to
this system \cite{abi,  abipai, AGZ2, AGZ3, AKM, danchin, danchin2,
dm, dm2, desj, ger, guizhang, lions}... Our main goal in this paper
is to provide a global wellposedness result for the
density-dependent incompressible Navier-Stokes equations with
variable viscosity, in a critical functional framework which is
invariant by the  scaling of the equations and under a non-linear
smallness condition on fluctuation of the initial density which has
to be doubly  exponential small compared with the size of the
initial velocity. In the second part of the paper, we apply our
methods combined with the techniques in \cite{dm} to prove the
global existence of the solution to \eqref{INS} with piecewise
constant initial density, which is away from vacuum and has small
jumps at the interface. This latter problem is of a great interest
from physical point of view as it represents the case of a
immiscible mixture of fluids with different densities. We give in
this manner a partial response of a question raised by Lions
 \cite{lions} concerning the propagation of the regularity of the boundary to a "density-patches".

 We briefly describe in this paragraph some of the classical
results for the inhomogeneous Navier-Stokes system. When the viscous
coefficient equals some positive constant,  Lady\v zenskaja and
Solonnikov \cite{LS} first established the unique resolvability of
(\ref{INS}) in a bounded domain $\Om$ with homogeneous Dirichlet
boundary condition for $u;$ similar result was obtained by Danchin
\cite{danchin2} in $\R^d$ with initial data in the almost critical
Sobolev spaces; Simon \cite{Simon} proved  the global existence of
weak solutions. In general,  the global existence of weak solutions
with finite energy to \eqref{INS} with variable viscosity was proved
by Lions in \cite{lions} (see also the references therein, and the
monograph \cite{AKM}). Yet the regularity and uniqueness of such
weak solutions is a big open question in the field of mathematical
fluid mechanics, even in two space dimensions when the viscosity
depends on the density. Except under the assumptions:
\[ \r_0\in L^\infty({\Bbb T}^2),\quad
\inf_{c>0}\Bigl\|\frac{\mu(\rho_0)}{c}-1\Bigr\|_{L^\infty({\Bbb
T}^2)}\leq \epsilon,\quad \mbox{and}\quad u_0\in H^1({\Bbb T}^2),\]
Desjardins \cite{desj} proved that Lions weak solution $(\r, u)$
satisfies $u\in L^\infty((0,T); H^1({\Bbb T}^2))$ and $\rho\in
L^\infty((0,T)\times{\Bbb T}^2)$ for any $T<\infty.$  Moreover, with
additional assumption on the initial density, he could also prove
that $u\in L^2((0,\tau); H^2({\Bbb T}^2))$ for some short time
$\tau.$ To understand this problem further, the third author to this
paper proved the global wellposedness to a modified 2-D model
problem of \eqref{INS}, which coincides with the 2-D inhomogeneous
Navier-Stokes system with  constant viscosity, with general initial
data in \cite{Zhang}. Gui and Zhang \cite{guizhang} proved the
global wellposedness of \eqref{INS} with initial data satisfying
$\|\rho_0-1\|_{H^{s+1}}$ being sufficiently small and $u_0\in
H^s(\R^2)\cap \dot{H}^{-\e}(\R^2)$ for some $s>2$ and $0<\e<1.$
However, the exact size of $\|\rho_0-1\|_{H^{s+1}}$ was not given in
\cite{guizhang}.

Very recently,  Danchin and Mucha \cite{dm2} proved that: given
initial density $\r_0$ in  $ L^\infty(\Om)$ with a positive lower
bound and initial velocity $u_0\in H^2(\Om)$ for some bounded smooth
domain of $\R^d,$ the system \eqref{INS} with constant viscosity
 has a unique local solution. Furthermore, with the
initial density being close enough to some positive constant, for
any initial velocity in two space dimensions, and sufficiently small
velocity in three space dimensions, they also proved its global
wellposedness. We remark that the Lagrangian formulation for the
describing the flow plays a key role in the analysis in \cite{dm2}.
To prove the 2-D global result, they first applied energy method to
obtain $L^\infty(\R^+;H^1(\Om))$ estimate for the velocity field $u$
and $L^2(\R^+;L^2(\Om))$ estimate for $\p_tu.$ Then the authors
employed the classical maximal $L^p_T(L^q)$ estimate for the linear
Stokes operator to obtain the second order space derivative estimate
for the velocity. Notice that when  $\mu(\r)$ depends on $\r,$ and
the initial density is sufficiently close to some positive constant
in $L^\infty(\R^2),$ one can recover $L^\infty(\R^+;H^1(\R^2))$
estimate for the velocity $u$ and $L^2(\R^+;L^2(\R^2 ))$ estimate
for $\p_tu$  by using Desjardins' techniques from \cite{desj}. Yet
we do not know then how to recover the second order space
derivatives of the velocity. Therefore, I think it is a very
challenging problem to prove Danchin and Mucha \cite{dm2} type
results for \eqref{INS} with variable viscosity.

When the density $\rho$ is away from zero, we denote by $a \eqdefa
\f{1}{\rho}-1$ and $\wt{\mu}(a)\eqdefa \mu(\f1{1+a})$, then the
system \eqref{INS} can be equivalently reformulated as
\begin{equation}
\label{INS1}
 \left\{\begin{array}{l}
\displaystyle \pa_t a + u \cdot \na a=0,\qquad (t,x)\in\R^+\times\R^2,\\
\displaystyle \pa_t u + u \cdot \na u +(1+a)(\na \Pi-\mathrm{div}(\wt{\mu}(a)\cM)=0, \\
\displaystyle \mathrm{div} u = 0.\\
\end{array}\right.
\end{equation}
 Notice that just as the
classical Navier-Stokes system (which corresponds to $a=0$ in
\eqref{INS1}), the inhomogeneous Navier-Stokes system (\ref{INS1})
also has a scaling. In fact, if $(a, u)$ solves (\ref{INS1}) with
initial data $(a_0, u_0)$, then for any $ \ell>0$,
\begin{equation}\label{1.2}
(a, u)_{\ell} \eqdefa (a(\ell^2\cdot, \ell\cdot), \ell u(\ell^2
\cdot, \ell\cdot))\quad\mbox{and}\quad (a_0,u_0)_\ell\eqdefa
(a_0(\ell\cdot),\ell u_0(\ell\cdot))
\end{equation}
 $(a, u)_{\ell}$ is also a solution of (\ref{INS1}) with initial data $(a_0,u_0)_\ell$.

It is easy to check that the norm of
$B_{p,1}^{\frac{d}p}(\R^d)\times B_{p,1}^{-1+\frac{d}p}(\R^d)$ is
scaling invariant under the scaling transformation $(a_0,u_0)_\ell$
given by \eqref{1.2}. In \cite{abi}, Abidi proved in general space
dimension $d$ that: if $1<p<2d,$ $ 0<\underline{\mu}<\mu(\r),$ given
$a_0\in B_{p,1}^{\frac{d}p}(\R^d)$ and  $u_0\in
B_{p,1}^{-1+\frac{d}p}(\R^d),$ (\ref{INS1}) has a global solution
provided that
$\|a_0\|_{B_{p,1}^{\frac{d}p}}+\|u_0\|_{B_{p,1}^{-1+\frac{d}p}}\leq
c_0$ for some sufficiently small $c_0$. Moreover, this solution is
unique if $1<p\leq d.$  This result generalized the corresponding
results in \cite{danchin, danchin2} and was improved by Abidi and
Paicu in \cite{abipai} with  $a_0\in B_{q,1}^{\frac{d}q}(\R^d)$ and
$u_0\in B_{p,1}^{-1+\frac{d}p}(\R^d)$ for $p,q$ satisfying some
technical assumptions.  Abidi, Gui and Zhang  removed the smallness
condition for $a_0$ in \cite{AGZ2, AGZ3}. Notice that  the main
feature of the density space is to be a multiplier on the velocity
space and this allows to define the nonlinear terms in the system
\eqref{INS1}. Recently, Danchin and Mucha \cite{dm} proved a more
general wellposedness result of \eqref{INS} with $\mu(\r)=\mu>0$ by
considering very rough densities in some multiplier spaces on the
Besov spaces $B_{p,1}^{-1+\frac{d}p}(\R^d)$ for $1<p<2d,$ which in
particular completes the uniqueness result  in \cite{abi} for $p\in
(d,2d)$ in the case when $\mu(\r)=\mu>0.$

On the other hand, motivated by \cite{GZ2, PZ1, Zhangt} concerning
the global wellposedness of 3-D incompressible anisotropic
Navier-Stokes system with the third component of the initial
velocity field being large, we \cite{PZ2} proved that:  given
$a_0\in B^{\f3q}_{q,1}(\R^3)$ and $u_0=(u_0^h,u_0^3)\in
B^{-1+\frac3p}_{p,1}(\R^3)$ for $1<q\leq p<6$ and
$\frac1q-\frac1p\leq \frac13,$ \eqref{INS1} with $\wt{\mu}(a)=\mu>0$
has a unique global solution  as long as
$$\bigl(\mu\|a_0\|_{B_{q,1}^{\f3q}}+\|u_0^h\|_{B^{-1+\frac3p}_{p,1}}\bigr)\exp\Bigl\{
C_0\|u_0^3\|_{B^{-1+\frac3p}_{p,1}}^2\ \Big/\mu^2\Bigr\}\leq
c_0\mu$$ for some  sufficiently small $c_0$. We emphasize that our
proof in \cite{PZ2} used in a fundamental way the algebraical
structure of \eqref{INS1}, namely $\dive u=0.$

The first object of this paper is to improve the global
wellposedness  result in \cite{guizhang} so that given initial data
in the scaling invariant Besov spaces, for any initial velocity,
\eqref{INS1} has a global solution provided that the fluctuation of
the initial density is sufficiently small, furthermore, its explicit
dependence on the initial velocity  will be given here.

\begin{thm}\label{mainthm}
{\sl Let $1 < q \leq p <4$, and $\f1q -\f1p \leq \f12$. Let $a_0 \in
B^{\f2q}_{q,1}(\mathbb{R}^2)$ and $u_0 \in
B^{-1+\f2p}_{p,1}(\mathbb{R}^2)$ be a solenoidal vector field. Then
there exist positive constants $c_0$ and $C_0,$ which depend on
$\|\wt{\mu}'\|_{L^\infty(-1,1)},$ such that if
\begin{equation}\label{thmassume}
\eta\eqdefa
\normbq{a_0}{\f2q}\exp\Bigl\{C_0\bigl(1+\wt{\mu}^2(0)\bigr)\exp\bigl(\f{C_0}{\wt{\mu}^2(0)}\normbp{u_0}{-1+\f2p}^2\bigr)\Bigr\}
\leq \f{c_0\wt{\mu}(0)}{1+\wt{\mu}(0)},
\end{equation}
\eqref{INS1} has a global solution $a\in
\cC([0,\infty);B^{\f2q}_{q,1}(\R^2))\cap
\wt{L}^\infty(\mathbb{R}^+;B^{\f2q}_{q,1}(\R^2))$ and $u \in
\cC([0,\infty);B^{-1+\f2p}_{p,1}(\R^2))\cap
\wt{L}^\infty(\mathbb{R}^+; {B}^{-1+\f2p}_{p,1}(\R^2))\cap
L^1(\mathbb{R}^+;B^{1+\f2p}_{p,1}(\R^2))$. If $\f1p+\f1q\geq 1,$
this solution is unique. }
\end{thm}

\begin{rmk}
\begin{itemize}
\item The definitions of the functional spaces will be presented in
Subsection \ref{subsect2.1}.

\item We remark that compared with the finite energy solutions constructed in \cite{dm2},
 our solution here is not of finite energy and belongs to the critical spaces related to \eqref{INS1}. While for the  classical
2-D Navier-Stokes system,  large infinite energy solution was proved
by Gallagher and Planchon  \cite{gallplan} and Germain \cite{ger}.
\end{itemize}
\end{rmk}

It turns out that we can apply the main idea  to prove Theorem
\ref{mainthm} combined with the techniques in \cite{dm} to remove
the smallness condition for initial velocity in \cite{dm} when the
space dimension equals to $2.$ Toward this,  we first recall the
definition of multiplier spaces to Besov spaces from \cite{MS}:

\begin{defi}\label{deffm}
The multiplier space $\cM(B^{s}_{p,1}(\R^d))$ of $B^{s}_{p,1}(\R^d)$
is  the set of  distributions $f$ such that $ f \psi\in
B^{s}_{p,1}(\R^d)$ whenever $\psi\in B^{s}_{p,1}(\R^d)$. We endow
this space with the norm \beno \|f\|_{\cM(B^{s}_{p,1})}\eqdefa
\sup_{\psi\in B^{s}_{p,1}(\R^d):\ \|\psi\|_{B^{s}_{p,1}}\leq
1}\|\psi f\|_{B^{s}_{p,1}}.\eeno
\end{defi}

In \cite{dm},  Danchin and Mucha proved the following global
wellposedness for \eqref{INS} with constant viscosity:

\begin{thm}\label{thdm1}
(Theorem  1 and Theorem 3 of \cite{dm}) {\sl Let $p\in [1,2d)$ and
$u_0$ be a divergence-free vector field in
$B^{-1+\f{d}{p}}_{p,1}(\R^d).$ Assume that the initial density
$\r_0$ belongs to the multiplier space
$\cM(B^{-1+\f{d}p}_{p,1}(\R^d)).$ There exists a constant $c$
depending only on $d$ such that if
$$
\|\r_0-1\|_{\cM(B^{-1+\f{d}p}_{p,1})}+\mu^{-1}\|u_0\|_{B^{-1+\f{d}{p}}_{p,1}}\leq
c,
$$ system \eqref{INS} with $\mu(\r)=\mu>0$ has a unique global
solution $(\r, u)$ with $\r\in
L^\infty(\R^+;\cM(B^{-1+\f{d}p}_{p,1}(\R^d)))$ and $u \in
\cC([0,\infty); B^{-1+\f{d}{p}}_{p,1}(\R^d))$ $ \cap L^1(\R^+;
B^{1+\f{d}{p}}_{p,1}(\R^d)).$   }
\end{thm}

Motivated by \cite{dm} and the proof of Theorem \ref{mainthm}, here
we consider similar global wellposedness  of \eqref{INS1}, which
does not require any smallness assumption for $u_0.$

\begin{thm}\label{mainthm1}
{\sl Let $p\in (2,4),$ $a_0\in \cM(B^{-1+\f2p}_{p,1}(\R^2))$ with
$\wt{\mu}(a_0)\in \cM(B^{\f2p}_{p,1}(\R^2)),$ and $u_0 \in
B^{-1+\f2p}_{p,1}(\mathbb{R}^2).$  Then there exist positive
constants $c_0$ and $C_0$ such that if
\begin{equation}\label{thmassumea}
\begin{split}
(\mu\|a_0\|_{\cM(B^{-1+\f2p}_{p,1})}&+\|\wt{\mu}(a_0)-\wt{\mu}(0)\|_{\cM(B^{\f2p}_{p,1})})\\
&\times\exp\Bigl\{C_0\bigl(1+\wt{\mu}^2(0)\bigr)\exp\bigl(\f{C_0}{\wt{\mu}^2(0)}\normbp{u_0}{-1+\f2p}^2\bigr)\Bigr\}
\leq c_0\wt{\mu}(0),
\end{split}
\end{equation}
\eqref{INS1} has a unique global solution $(a, u)$ with $a\in
L^\infty(\R^+; \cM(B^{-1+\f2p}_{p,1}(\R^2))$ and
 $u \in
\cC([0,\infty); B^{-1+\f{2}{p}}_{p,1}(\R^2))$ $ \cap L^1(\R^+;
B^{1+\f{2}{p}}_{p,1}(\R^2)).$ }
\end{thm}

Notice from \cite{dm} that: let $\Omega_0$ be a bounded $C^1$ domain
of $\R^2$ and $\rho_0 =1 +\sigma \chi_{\Omega_0}$ for some
sufficiently  small constant $\s$, $a_0=\f1{\r_0}-1=
-\f{\sigma}{1+\sigma}\chi_{\Omega_0}$  and
$\wt{\mu}(a_0)-\wt{\mu}(0)=(\mu(1+\sigma)-\mu(1))\chi_{\Omega_0}$
belong to $\cM(B^{-1+\f2p}_{p,1}(\R^2))$ for  $2\leq p <4$ and their
$\cM(B^{-1+\f2p}_{p,1}(\R^2))$ norm are small as long as $|\sigma|$
is small. This together with Theorem \ref{mainthm1} implies that

\begin{col}\label{col1.1}
{\sl Let $p\in (2,4)$ and $u_0 \in B^{-1+\f2p}_{p,1}(\mathbb{R}^2)$
be a solenoidal vector field. Let $\Omega_0$ be a bounded $C^1$
domain of $\R^2$ and $\rho_0 =1 +\sigma \chi_{\Omega_0}$ for some
small enough constant $\sigma$ (compared to $\|u_0\|_{
B^{-1+\f2p}_{p,1}}$).
 Then
\eqref{INS} has a unique global solution $(\rho, u)$ with $u \in
\cC([0,\infty); B^{-1+\f{2}{p}}_{p,1}(\R^2)) \cap
L^1(\R^+;B^{1+\f{2}{p}}_{p,1}(\R^2))$ and $$
\rho(t)=1+\sigma\chi_{\Omega_t} \quad \mbox{for}\quad \Omega_t =
X_u(t, \Omega_0),
$$
where $X_u(t,y)$ is determined by \beq\label{flow}
X_u(t,y)=y+\int_0^t u(\tau, X_u(\tau,y))d\tau. \eeq Besides, the
measure of $\Om_t$ and the $C^1$ regularity of $\pa\Omega_t$ are
preserved for all time.}
\end{col}

\begin{rmk} \begin{itemize}
\item  We have considered here the physical case of a density given by a discontinuous function (immiscible fluids)
and of a viscous coefficient depending on the density of the fluid.
In particular, our  Corollary \ref{col1.1} removes the smallness
condition for the initial velocity field in  Corollary 1 of
\cite{dm}.  In fact,  for $\Om_0$ being a  bounded $C^2$ domain of
$\R^2$ and $u_0\in L^2(\Om)\cap B^1_{4,2}(\Om)$ (which is above the
critical regularity of \eqref{INS}), Danchin and Mucha \cite{dm2}
can prove a similar global wellposedness result for \eqref{INS} with
constant viscosity.

\item Given initial data $a_0,u_0$ in the scaling invariant spaces: $a_0\in L^\infty(\R^d)$ and $u_0\in
B^{-1+\f{d}p}_{p,r}(\R^d)$ for $1<p<d, 1<r<\infty,$ and which
satisfies some nonlinear smallness condition, we \cite{HPZ3} proved
that \eqref{INS1} with $\wt{\mu}(a)=\mu>0$ has a  global weak
solution. And the uniqueness of such solution is in progress.
\end{itemize}
\end{rmk}

\no{\bf Scheme of the proof and organization of the paper.} The
strategy to the  proof of both Theorem \ref{mainthm} and Theorem
\ref{mainthm1} is to seek a solution of \eqref{INS1} with the form
$u=v+ w$ with $(w,p)$ solving the classical Navier-Stokes system
\begin{equation}\label{weq}
\left\{\begin{array}{l}
 \pa_t w + w\cdot \na w - \mu\tri w +
\na p = 0,\qquad (t,x)\in\R^+\times\R^2, \\
\dv w=0,\\
 w|_{t=0} = u_0,
 \end{array}\right.
\end{equation}
and $(a,v, \Pi_1)$ solving \beq \label{veq} \left\{\begin{array}{l}
\pa_t a + (v + w)\cdot \na a = 0,\qquad (t,x)\in\R^+\times\R^2,\\
\pa_t v + v\cdot \na v + w\cdot \na v + v\cdot \na w - (1+a)\dive (\tilde{\mu}(a)\cM(v)) + (1+a)\na \Pi_1\\
=(1+a)\dive[(\tilde{\mu}(a)-\mu)\cM(w)] + \mu a\tri w -a\na p \eqdef F,\\
\dive v=0,\\
(a, v)|_{t=0} = (a_0, 0),
\end{array}\right.
\eeq where and in what follows, we shall always denote $\wt{\mu}(0)$
by $\mu.$

In Section 2, we shall first collect some basic facts on
Littlewood-Paley theory, and then present the  estimates to the free
transport equation and the pressure function determined by
\eqref{veq}; in Section 3, we solve \eqref{weq} for $w$ with $u_0\in
B^{-1+\f2p}_{p,1}(\R^2)$ for $1<p<4.$ We should mention that because
of the restriction to the index $p$ in $(1,4),$ the proof here is
much simpler than that in \cite{gallplan,ger}. Then we prove Theorem
\ref{mainthm} in Section 4. Finally along the same lines to the
proof of Theorem \ref{mainthm},
we present the proof of Theorem \ref{mainthm1} in the last section.\\

Let us complete this introduction by the notations we shall use in
this context.

\no\textbf{Notation.} Let $A$,$B$ be two operators, we denote
$[A;B]=AB-BA$, the commutator between $A$ and $B$.
 For $a \lesssim b$, we mean that there is a uniform constant $C$, which may be different on different lines,
 such that $a \leq Cb$. We shall denote by $(a\ |\ b)$ the $L^2$ inner product of $a$ and $b$. $(d_j)_{j \in \mathbb{Z}}$
will be a generic element of $\ell^1(\mathbb{Z})$ so that
$\sum_{j\in \mathbb{Z}}d_j = 1$. For $X$ a Banach space and $I$ an
interval of $\R,$ we denote by ${\cC}(I;\,X)$ the set of continuous
functions on $I$ with values in $X,$ and by  $L^q(I;\,X)$ stands for
the set of measurable functions on $I$ with values in $X,$ such that
$t\longmapsto\|f(t)\|_{X}$ belongs to $L^q(I).$

\bigskip
\renewcommand{\theequation}{\thesection.\arabic{equation}}
\setcounter{equation}{0}
\section{Preliminary Estimates}

\subsection{Some Basic Facts on Littlewood-Paley Theory} \label{subsect2.1} For the convenience of the readers, we recall the following basic
facts on Littlewood-Paley theory from \cite{BCD}: for $a\in{\mathcal
S}'(\R^2),$ we set
 \beq
\begin{split}
&\Delta_ja\eqdefa\cF^{-1}(\varphi(2^{-j}|\xi|)\widehat{a}),
 \qquad\ \
S_ja\eqdefa\cF^{-1}(\chi(2^{-j}|\xi|)\widehat{a}),
\end{split} \label{1.0}\eeq where $\cF a$ and $\widehat{a}$ denote
the Fourier transform of the distribution $a,$ ~$\varphi(\tau)$ and
$\chi(\tau)$ are  smooth functions such that \beno
\begin{split}
&\Supp \varphi \subset \Bigl\{\tau \in \R\,/\  \ \frac34\leq |\tau|
\leq \frac83 \Bigr\}\andf \  \ \forall
 \tau>0\,,\ \sum_{j\in\Z}\varphi(2^{-j}\tau)=1,\\
&\Supp \chi \subset \Bigl\{\tau \in \R\,/\  \ \ |\tau| \leq \frac43
\Bigr\}\quad \ \ \andf \  \ \, \chi(\tau)+ \sum_{j\geq
0}\varphi(2^{-j}\tau)=1.
\end{split}
 \eeno
We have the formal decomposition
\begin{equation*}
u=\sum_{j\in\Z}\Delta_j \,u,\quad\forall\,u\in {\mathcal
{S}}'(\R^2)/{\mathcal{P}}[\R^2],
\end{equation*}
where ${\mathcal{P}}[\R^2]$ is the set of polynomials (see
\cite{PE}). Moreover, the Littlewood-Paley decomposition satisfies
the property of almost orthogonality:
\begin{equation}\label{Pres_orth}
\Delta_j\Delta_k u\equiv 0 \quad\mbox{if}\quad| j-k|\geq 2
\quad\mbox{and}\quad\Delta_j(S_{k-1}u\Delta_k v) \equiv
0\quad\mbox{if}\quad| j-k|\geq 5.
\end{equation}

\begin{defi}\label{def1.1}[Definition 2.15 of \cite{BCD}]
{\sl  Let $(p,r)\in[1,+\infty]^2,$ $s\in\R.$ The homogeneous Besov
space $B^s_{p,r}(\R^2)$ consists of those distributions
$u\in{\mathcal S}_h'(\R^2),$ which means that $u\in {\mathcal
S}'(\R^2)$ and $\lim_{j\to-\infty}\|S_ju\|_{L^\infty}=0$ (see
Definition 1.26 of \cite{BCD}), such that
$$
\|u\|_{B^s_{p,r}}\eqdefa\Big(2^{qs}\|\Delta_q u\|_{L^{p}}\Big)_{\ell
^{r}(\Z)}<\infty. $$}
\end{defi}

In  order to obtain a better description of the regularizing effect
to the transport-diffusion equation, we will use Chemin-Lerner type
spaces $\widetilde{L}^{\lambda}_T(B^s_{p,r}(\R^2))$ (see \cite{BCD}
for instance).
\begin{defi}\label{chaleur+}
Let $(r,\lambda,p)\in[1,\,+\infty]^3$ and $T\in (0,\,+\infty]$. We
define $\widetilde{L}^{\lambda}_T(B^s_{p\,r}(\R^2))$ as the
completion of $C([0,T];\cS(\R^2))$ by the norm
$$
\| f\|_{\widetilde{L}^{\lambda}_T(B^s_{p,r})} \eqdefa
\Big(\sum_{q\in\Z}2^{qrs} \Big(\int_0^T\|\Delta_q\,f(t)
\|_{L^p}^{\lambda}\, dt\Big)^{\frac{r}{\lambda}}\Big)^{\frac{1}{r}}
<\infty,
$$
with the usual change if $r=\infty.$  For short, we just denote this
space by $\widetilde{L}^{\lambda}_T(B^s_{p,r}).$
\end{defi}

We also need the following form of functional framework, which is a
sort of  generalization to the weighted Chemin-Lerner type norm
defined \cite{PZ1, PZ2}:

\begin{defi}\label{defpz}
Let $f(t)\in L^1_{loc}(\R^+)$, $f(t)\geq 0$ and $X$ be a Banach
space. We define
$$\|u\|_{L^1_{T,f}(X)}\eqdefa \int_0^Tf(t)\|
u(t)\|_{X}\,dt. $$
\end{defi}

\begin{lem}\label{bernlem}
{\sl Let $\cB$ be a ball of $\mathbb{R}^2$, and $\cC$ be a ring of $\mathbb{R}^2$; let $1\leq p_2 \leq p_1 \leq \infty$. Then there hold:\\
If the support of $\hat a$ is included in $2^k \cB$, then \beno
\|\pa^\al_{x}a\|_{L^{p_1}} \lesssim
2^{k(|\al|+2(\f{1}{p_2}-\f{1}{p_1}))}\|a\|_{L^{p_2}}. \eeno If the
support of $\hat a$ is included in $2^k \cC$, then \beno
\|a\|_{L^{p_1}} \lesssim 2^{-kN}\sup_{|\al|=N}\|\pa_x^\al
a\|_{L^{p_1}}. \eeno}
\end{lem}

\begin{lem}\label{action}
{\sl Let $\theta$ be a smooth function supported in an annulus $\cC$
of $\R^d$. There exists a constant $C$ such that for any $C^{0,1}$
measure-preserving global diffeomorphism $\psi$ over $\R^d$ with
inverse $\phi$, any tempered distribution $u$ with $\hat{u}$
supported in $\la\cC$, any $p\in[1,\infty]$ and any $(\la, \mu)\in
(0,\infty)^2$, we have \beno
\|\theta(\mu^{-1}D)(u\circ\psi)\|_{L^p}\leq
C\|u\|_{L^p}\min(\f{\mu}{\la}\|\na\phi\|_{L^\infty},\f{\la}{\mu}\|\na
\psi\|_{L^\infty}). \eeno}
\end{lem}

\begin{lem}\label{heatkernel}
{\sl  If the support of $\hat u$ is included in $\la \cC$, then
there exists a positive constant $c$, such that \beno
\|e^{t\tri}u\|_{L^p}\lesssim e^{-c\la^2t}\|u\|_{L^p}\quad \mbox{for
any}\quad p\in[1,\infty]. \eeno}
\end{lem}

\begin{lem}
{\sl Let $p_1 \geq p_2 \geq 1$, and $s_1 \leq \f{2}{p_1}$, $s_2 \leq
\f{2}{p_2}$ with $s_1 + s_2 > 0$. Let $a \in
B^{s_1}_{p_1,1}(\mathbb{R}^2)$, $b \in
B^{s_2}_{p_2,1}(\mathbb{R}^2)$. Then $ab \in
B_{p_1,1}^{s_1+s_2-\f{2}{p_1}}(\mathbb{R}^2)$ and \beno
\|ab\|_{B_{p_1,1}^{s_1+s_2-\f{2}{p_1}}} \lesssim
\|a\|_{B^{s_1}_{p_1,1}} \|b\|_{B^{s_2}_{p_2,1}}. \eeno}
\end{lem}

\begin{prop}\label{stokesprop}
{\sl Let $p\in (1,\infty)$, $r\in[1,\infty]$ and $s\in \mathbb{R}$.
Let $u_0\in B^s_{p,r}(\R^2)$ be a divergence-free field and
$g\in\wt{L}^1_T(B^s_{p,r})$. Then the following system
\begin{equation*}\label{stokes}
 \left\{\begin{array}{l}
\displaystyle \pa_t u -\nu\tri u+\na \Pi=g,\qquad (t,x)\in\R^+\times\R^2,\\
\displaystyle \dive u = 0,\\
\displaystyle u|_{t=0}=u_0,
\end{array}\right.
\end{equation*}
 has a unique solution $(u, \na\Pi)$ so that  \beno
\|u\|_{\wt{L}^\infty_T(B^s_{p,r})}
+\mu\|u\|_{\wt{L}^1_T(B^{s+2}_{p,r})}+\|\na\Pi\|_{\wt{L}^1_T(B^s_{p,r})}
\leq C\bigl(\|u_0\|_{B^s_{p,r}} +
\|g\|_{\wt{L}^1_T(B^s_{p,r})}\bigr). \eeno}
\end{prop}

\subsection{Estimates of the transport equation} The goal of  this
section is to investigate the transport equation in \eqref{veq} \beq
\label{transeq} \pa_t a + (v+w) \cdot \na a =0, \qquad a|_{t=0} =
a_0. \eeq More precisely, we shall prove the following proposition:
\begin{prop}\label{aprop}
{\sl Let $1 < q \leq p $ with $\f1q -\f1p \leq \f12$. Let $v,w \in
L^1((0,T),B^{1+\f2p}_{p,1}(\R^2))$ be divergence free vector fields,
  and $a_0 \in B^{\f2q}_{q,1}(\R^2)$. We denote $f(t)\eqdefa\|w(t)\|_{B^{1+\f2p}_{p,1}}$
  and $a_\la \eqdefa a\exp\bigl\{-\la\int_0^tf(\tau)d\tau\bigr\}$.
 Then \eqref{transeq} has a unique solution $a
\in {\cC}([0,T]; B^{\f2q}_{q,1}(\R^2))$ so that \beq \label{aest}
\|a_\la\|_{\wt{L}^\infty_t(B^{\f2q}_{q,1})}
+\f{\la}{2}\|a_\la\|_{L^1_{t,f}(B^{\f2q}_{q,1})}\leq
\|a_0\|_{B^{\f2q}_{q,1}}+
C\|v\|_{L^1_t(B^{1+\f2p}_{p,1})}\|a_\la\|_{\wt{L}^\infty_t(B^{\f2q}_{q,1})}
\eeq for any $t\in(0,T]$ and $\la$ large enough, and where
$\|a_\la\|_{L^1_{t,f}(B^{\f2q}_{q,1})}$ is given by Definition
\ref{defpz}.}
\end{prop}

\begin{proof} As both the existence and uniqueness of solution to \eqref{transeq}
basically follows from \eqref{aest}. For simplicity, we just present
the {\it a priori} estimate  \eqref{aest} for smooth enough
solutions of \eqref{transeq}. Indeed
 thanks to \eqref{transeq}, we have \beqo \pa_t
a_{\la} + \lam f(t)a_{\la}  + (v+w)\cdot \na a_{\lam}= 0. \eeqo
Applying $\D_j$ to the above equation and taking the $L^2$ inner
product of the resulting equation with $|\D_j a_\la|^{q-2}\D_j
a_{\lam}$ (in the case when $q\in (1,2),$ we need a small
modification to make this argument rigorous, which we omit here), we
obtain \beq \label{3.1} \frac1q\frac{d}{dt}\|\D_j
a_{\lam}(t)\|_{L^q}^q + \lam f(t)\|\D_j a_{\lam}(t)\|_{L^q}^q +
\bigl(\D_j((v+w)\cdot \na a_{\lam})\ \bigl|\ |\D_j
a_{\lam}|^{q-2}\D_ja_{\lam}\bigr)=0. \eeq Applying Bony's
decomposition to $(v+w)\cdot \na a_{\lam}$ gives rise to \beqo
(v+w)\cdot \na a_{\lam} = T_{(v+w)} \na a_{\lam} + T_{\na
a_\la}(v+w)+R((v+w),\na a_{\lam}). \eeqo One gets by using a
standard commutator argument that \beno
\begin{split}
\bigl(&\D_j(T_{(v+w)} \na a_{\lam})\ \bigl|\ |\D_j
a_{\lam}|^{q-2}\D_j a_{\lam}\bigr)\\
& = \sum_{|j-j'|\leq 5}\Bigl\{ \bigl([\D_j ;S_{j'-1}(v+w)]\D_{j'}\na
a_{\lam}\ \bigl|\ |\D_j
a_{\lam}|^{q-2}\D_j a_{\lam}\bigr)\\
&\quad+ \bigl((S_{j'-1}(v+w)-S_{j-1}(v+w))\D_j\D_{j'}\na a_{\lam}\
\bigl|\ |\D_j a_{\lam}|^{q-2}\D_j a_{\lam}\bigr)\Bigr\}\\
&\quad+\bigl(S_{j-1}(v+w)\na\D_j a_{\lam}\ \bigl|\ |\D_j
a_{\lam}|^{q-2}\D_j a_{\lam}\bigr),
\end{split} \eeno
 as $\dive v=\dive w=0,$ the last term equals $0,$
 from which and \eqref{3.1}, we infer
\beq\label{3.2}
\begin{split}
&\|\D_ja_\la (t)\|_{L^q} + \la \int_0^tf(\tau)\|\D_ja_\la (\tau)\|_{L^q}d\tau \\
&\leq\|\D_ja_0\|_{L^q}+C\Bigl\{\sum_{|j-j'|\leq 5}\bigl(\|[\D_j ;S_{j'-1}(v+w)]\D_{j'}\na a_{\lam}\|_{L^1_t(L^q)} \\
&\quad+\|(S_{j'-1}(v+w)-S_{j-1}(v+w))\D_j\D_{j'}\na a_{\lam}\|_{L^1_t(L^q)}\bigr)\\
&\quad+\|T_{\na a_\la}(v+w)\|_{L^1_t(L^q)}+\|R((v+w),\na
a_{\lam})\|_{L^1_t(L^q)}\Bigr\}.
\end{split}
\eeq We first get by applying the classical estimate on commutator
(see \cite{BCD} for instance) and Definition \ref{defpz} that \beno
\begin{split}
&\sum_{|j-j'|\leq 5}\|[\D_j ;S_{j'-1}(v+w)]\D_{j'}\na a_{\lam}\|_{L^1_t(L^q)} \\
&\lesssim \sum_{|j-j'|\leq 5}\bigl(\|\na S_{j'-1}v\|_{L^1_t(L^\infty)}\|\D_{j'}a_{\lam}\|_{L^\infty(L^q)} +\int_0^t\|\na S_{j'-1}w(\tau)\|_{L^\infty}\|\D_{j'}a_{\lam}(\tau)\|_{L^q}d\tau\bigr)\\
&\lesssim\sum_{|j-j'|\leq 5}\bigl(d_{j'}2^{-j'\f2q}\normBpo{v}{1+\f2p}\normBq{a_\la}{\f2q}+\int_0^t\normbp{w(\tau)}{1+\f2p}\|\D_{j'}a_{\lam}(\tau)\|_{L^q}d\tau\bigr)\\
&\lesssim d_j
2^{-j\f2q}\bigl(\normBpo{v}{1+\f2p}\normBq{a_\la}{\f2q}+\|a_\la\|_{L^1_{t,f}(B^{\f2q}_{q,1})}\bigr).
\end{split}
\eeno  Applying Lemma \ref{bernlem} leads to  \beno
\begin{split}
&\sum_{|j-j'|\leq 5}\|(S_{j'-1}(v+w)-S_{j-1}(v+w))\D_j\D_{j'}\na a_{\lam}\|_{L^1_t(L^q)}\\
&\lesssim \sum_{|j-j'|\leq 5}\big(\|S_{j'-1}\na v-S_{j-1}\na v\|_{L^1_t(L^\infty)}\|\D_{j}a_{\lam}\|_{L^\infty(L^q)}\\
&\qquad+\int_0^t\|(S_{j'-1}\na w-S_{j-1}\na w)(\tau)\|_{L^\infty}\|\D_{j'}a_{\lam}(\tau)\|_{L^q}d\tau\bigr)\\
& \lesssim d_j 2^{-j\f2q}\normBpo{v}{1+\f2p}\normBq{a_\la}{\f2q}+\sum_{|j-j'|\leq 5}\int_0^t\normbp{w(\tau)}{1+\f2p}\|\D_{j'}a_{\lam}(\tau)\|_{L^q}d\tau\\
& \lesssim d_j
2^{-j\f2q}\bigl(\normBpo{v}{1+\f2p}\normBq{a_\la}{\f2q}+\|a_\la\|_{L^1_{t,f}(B^{\f2q}_{q,1})}\bigr).
\end{split}
\eeno
On the other hand, as $q \leq p$, let $r$ be determined by $\f1r = \f1q -\f1p$. Then we get that
\beno
\begin{split}
\|T_{\na a_\la}(v+w)\|_{L^1_t(L^q)} \lesssim &\sum_{|j-j'|\leq 5} \bigl(\|S_{j'-1}\na a_\la\|_{L_t^\infty(L^r)}\|\D_{j'}v\|_{L^1_t(L^p)}\\
&+\int_0^t\|S_{j'-1}\na a_\la(\tau)\|_{L^r}\|\D_{j'}w(\tau)\|_{L^p}d\tau\bigr),
\end{split}
\eeno
now as $\f1q-\f1p\leq \f12$ , one has \beno
\begin{split}
&\|S_{j'-1}\na a_\la\|_{L_t^\infty(L^r)} \lesssim \sum_{\ell \leq j'-2}2^{\ell(1+\f2p)}\|\D_\ell a_\la\|_{L_t^\infty(L^q)}\\
&\lesssim \sum_{l \leq j'-2}d_\ell
2^{\ell(1+\f2p-\f2q)}\normBq{a_\la}{\f2q} \lesssim
2^{j'(1+\f2p-\f2q)}\normBq{a_\la}{\f2q}.
\end{split}
\eeno Applying Lemma \ref{bernlem} and Definition \ref{defpz} once
again gives, if $\f1q-\f1p<\f12$ \beno
\begin{split}
\sum_{|j-j'|\leq 5}&\int_0^t\|S_{j'-1}\na a_\la(\tau)\|_{L^r}\|\D_{j'}w(\tau)\|_{L^p}d\tau\\
&\lesssim 2^{-j(1+\f2p)}\sum_{|j-j'|\leq 5}\sum_{\ell\leq j'-2}2^{\ell(1+\f2p)}\int_0^t\|\D_\ell a_\la(\tau)\|_{L^q}\normbp{w(\tau)}{1+\f2p}d\tau\\
&\lesssim 2^{-j(1+\f2p)}\sum_{\ell\leq j+3}d_\ell2^{\ell(1+\f2p-\f2q)}\|a_\la\|_{L^1_{t,f}(B^{\f2q}_{q,1})}\\
&\lesssim d_j 2^{-j\f2q}\|a_\la\|_{L^1_{t,f}(B^{\f2q}_{q,1})}.
\end{split}
\eeno
In the case when $\f1q-\f1p=\f12$, we have
\beno
\begin{split}
\sum_{j\in\mathbb{Z}}2^{j\f2q}&\sum_{|j-j'|\leq 5}\int_0^t\|S_{j'-1}\na a_\la(\tau)\|_{L^r}\|\D_{j'}w(\tau)\|_{L^p}d\tau\\
&\lesssim\sum_{j\in\mathbb{Z}}\sum_{\ell\leq j+3}2^{\ell(1+\f2p)}\int_0^td_j(\tau)\|\D_\ell a_\la(\tau)\|_{L^q}\normbp{w(\tau)}{1+\f2p}d\tau\\
&\lesssim \sum_{\ell\leq
j+3}d_\ell\|a_\la\|_{L^1_{t,f}(B^{\f2q}_{q,1})}\lesssim\|a_\la\|_{L^1_{t,f}(B^{\f2q}_{q,1})}.
\end{split}
\eeno As a consequence, we obtain \beno \|T_{\na
a_\la}(v+w)\|_{L^1_t(L^q)} \lesssim
d_j2^{-j\f2q}(\normBpo{v}{1+\f2p}\normBq{a_\la}{\f2q}+\|a_\la\|_{L^1_{t,f}(B^{\f2q}_{q,1})}).
\eeno For the last term in \eqref{3.2}, we deduce from Lemma
\ref{bernlem} that \beno
\begin{split}
\|R(v+w,\na a_{\lam})\|_{L^1_t(L^q)} \lesssim& 2^{\f{2j}{p}}\sum\limits_{j' \geq j-N_0}\bigl(\|\D_{j'}v\|_{L^1_t(L^p)}\|\tilde{\D}_{j'}\na a_\la\|_{L_t^\infty(L^q)}\\
&\qquad\qquad\quad+\int_0^t\|\D_{j'}w(\tau)\|_{L^p}\|\tilde{\D}_{j'}\na a_\la(\tau)\|_{L^q}d\tau\bigr)\\
\lesssim& 2^{\f{2j}{p}}\sum\limits_{j' \geq j-N_0}\bigl(d_{j'}2^{-j'(\f2p +\f2q)}\normBpo{v}{1+\f2p}\normBq{a_\la}{\f2q}\\
&\qquad\qquad\quad+2^{-j'\f2p}\int_0^t\normbp{w(\tau)}{1+\f2p}\|\tilde{\D}_{j'} a_\la(\tau)\|_{L^q}d\tau\bigr)\\
\lesssim& 2^{\f{2j}{p}}\sum\limits_{j' \geq j-N_0}d_{j'}2^{-j'(\f2p +\f2q)}\bigl(\normBpo{v}{1+\f2p}\normBq{a_\la}{\f2q} + \|a_\la\|_{L^1_{t,f}(B^{\f2q}_{q,1})}\bigr)\\
\lesssim& d_j2^{-j\f2q}\bigl(\normBpo{v}{1+\f2p}\normBq{a_\la}{\f2q}
+ \|a_\la\|_{L^1_{t,f}(B^{\f2q}_{q,1})}\bigr).
\end{split}
\eeno Substituting the above estimates into \eqref{3.2} and taking
summation for $j\in\mathbb{Z}$, we arrive at \beno
\normBq{a_\la}{\f2q} +\la\|a_\la\|_{L^1_{t,f}(B^{\f2q}_{q,1})}\leq
\normbq{a_0}{\f2q} + C\bigl(\normBpo{v}{1+\f2p}\normBq{a_\la}{\f2q}
+\|a_\la\|_{L^1_{t,f}(B^{\f2q}_{q,1})}\bigr). \eeno Taking $\la \geq
2C$ in the above inequality, we conclude the proof of \eqref{aest}.
\end{proof}

\subsection{Estimates of the pressure function}
In this subsection, we aim at providing the {\it a priori} estimate
for $\na \Pi_1$ determined by \eqref{veq}. We  first get by taking
$\dive$ to the momentum equation of \eqref{veq} that \beq
\label{presseq}
\begin{split}
-\tri \Pi_1 = &\dive(a\na \Pi_1) + \dive F - \dive(v\cdot \na v + w\cdot \na v + v\cdot \na w)\\
&+ \dive[(1+a)\dive\big((\tilde{\mu}(a)-\mu)\cM(v)\big)]
+\mu\dive(a\tri v),
\end{split}
\eeq where $F$ is given by \eqref{veq}.

\begin{prop}\label{pressprop}
{\sl Let $1 < q \leq p < 4$. Let $a \in
\wt{L}_T^\infty(B^{\f2q}_{q,1})$,
 $w,v \in L^1_T(B^{1+\f2p}_{p,1})\cap \widetilde{L}_T^\infty(B^{-1+\f2p}_{p,1})$ and $\na p \in L^1_T(B^{-1+\f2p}_{p,1})$.
For $\la_1,\la_2>0,$ we denote \beq\label{prop2.3}
\begin{split}
&f_1(t)\eqdefa\normbp{w(t)}{1+\f2p}+\f{1}{\mu}\normbp{\na
p(t)}{-1+\f2p},\quad
  f_2(t)\eqdefa\normbp{w(t)}{\f2p}^2\quad\mbox{ and }\\
  &\Pi_{\bar{\la}}\eqdefa \Pi_1\exp\Bigl\{-\la_1\int_0^tf_1(\tau)d\tau
 -\la_2\int_0^tf_2(\tau)d\tau\Bigr\},
 \end{split}
 \eeq
and similar notations for $a_{\bar{\la}}$ and $v_{\bar{\la}}$. Then
\eqref{presseq} has a unique solution $\na \Pi_1 \in
L^1_T(B^{-1+\f2p}_{p,1})$ so that for any $ \ep > 0,$ there holds
\beq \label{pressest}
\begin{split}
\normBpo{\na \Pi_{\bar{\la}}}{-1+\f2p} \leq &
\f{C}{1-C\normBq{a}{\f2q}} \Bigl\{\ep \normBpo{v_{\bar{\la}}}{1+\f2p} + \normBp{v}{-1+\f2p}\normBpo{v_{\bar{\la}}}{1+\f2p}\\
& +\|v_{\bar{\la}}\|_{L^1_{t,f_1}(B^{-1+\f2p}_{p,1})}+ \f{1}{\ep} \|v_{\bar{\la}}\|_{L^1_{t,f_2}(B^{-1+\f2p}_{p,1})}\\
& + \bigl(
\mu+\frak{C}(1+\normBq{a}{\f2q})\bigr)\bigl(\|a_{\bar{\la}}\|_{L^1_{t,f_1}(B^{\f2q}_{q,1})}+
\normBq{a}{\f2q}\normBpo{v_{\bar{\la}}}{1+\f2p}\bigr)\Bigr\}
\end{split}
\eeq provided that $C\normBq{a}{\f2q} \leq \f12,$  where
$\|v_{\bar{\la}}\|_{L^1_{t,f}(B^{-1+\f2p}_{p,1})}$ is given by
Definition \ref{defpz} and the positive constant $\frak{C}$ depends
on $\|\wt{\mu}'\|_{L^\infty(-1,1)}.$}
\end{prop}

The proof of this proposition will mainly be based on the following lemmas:

\begin{lem}\label{presslem1}
{\sl Under the assumptions of Proposition \ref{pressprop}, one has
for any $\ep>0$ \beno \normBpo{v\cdot\na w+w\cdot \na
v}{-1+\f2p}\lesssim \ep
\normBpo{v}{1+\f2p}+\|v\|_{L^1_{t,f_1}(B^{-1+\f2p}_{p,1})} +
\f{1}{\ep} \|v\|_{L^1_{t,f_2}(B^{-1+\f2p}_{p,1})}. \eeno}
\end{lem}

\begin{proof}
As $\dive v =\dive w =0$, we have \beno \normBpo{v\cdot\na w+w\cdot
\na v}{-1+\f2p} \lesssim \normBpo{vw}{\f2p}. \eeno While we get by
applying Bony's decomposition that \beno vw= T_vw + T_wv + R(v,w).
\eeno Notice that applying Lemma \ref{bernlem} leads to \beno
\begin{split}
\normop{\D_j( T_vw)} &\lesssim \int_0^t \sum_{|k-j|\leq 5}\normsup{S_{k-1}v}\normp{\D_k w}\,d\tau\\
&\lesssim \int_0^t d_j(\tau)
2^{-j(1+\f2p)}\normbp{w(\tau)}{1+\f2p}\sum_{\ell
\leq j}2^{\f2p \ell}\normp{\D_\ell v(\tau)}\,d\tau\\
& \lesssim d_j 2^{-j\f2p}\int_0^t
\normbp{v(\tau)}{-1+\f2p}\normbp{w(\tau)}{1+\f2p}\,d\tau,
\end{split}
\eeno and \beno
\begin{split}
\normop{\D_j( R(v,w))} & \lesssim  2^{j\f2p}\int_0^t \sum_{k\geq j-N_0}\normp{\wt{\D}_k v}\normp{\D_k w}\,d\tau\\
&\lesssim 2^{j\f2p}\int_0^t \sum_{k\geq j-N_0}d_k(\tau)
2^{-k\f4p}\normbp{v(\tau)}{-1+\f2p}\normbp{w(\tau)}{1+\f2p}\,d\tau\\
& \lesssim  d_j
2^{-j\f2p}\int_0^t\normbp{v(\tau)}{-1+\f2p}\normbp{w(\tau)}{1+\f2p}\,d\tau.
\end{split}
\eeno It follows from the same line that \beno
\begin{split}
\normop{\D_j (T_wv)}
&\lesssim \int_0^t \sum_{|k-j|\leq 5}\normsup{S_{k-1}w}\normp{\D_k v}\,d\tau
 \lesssim \sum_{|k-j|\leq 5}\int_0^t \normbp{w}{\f2p}\normp{\D_k v}\,d\tau\\
&\lesssim \Bigl(\sum_{|k-j|\leq 5}\int_0^t2^{-k}\normbp{w}{\f2p}^2\normp{\D_k v}\,d\tau\Bigr)^{\f12}
\Bigl(\sum_{|k-j|\leq 5}\int_0^t2^k\normp{\D_k v}\,d\tau\Bigr)^{\f12}\\
&\lesssim d_j 2^{-j\f2p}\Bigl(\f{1}{\ep}
\int_0^t\normbp{w(\tau)}{\f2p}^2 \normbp{v(\tau)}{-1+\f2p}\,d\tau +
\ep \int_0^t \normbp{v(\tau)}{1+\f2p}\,d\tau\Bigr).
\end{split}
\eeno The above estimates together with Definition \ref{defpz} prove
the lemma.
\end{proof}

\begin{lem}\label{presslem2}
{\sl Let $F$ be determined by \eqref{veq}. Then under the
assumptions of Proposition \ref{pressprop}, one has \beno
\normBpo{F}{-1+\f2p}\leq \bigl(C
\mu+\frak{C}(1+\normBq{a}{\f2q})\bigr)\|a\|_{L^1_{t,f_1}(B^{\f2q}_{q,1})}.
\eeno for some positive constant $\frak{C}$ depending on
$\|\wt{\mu}'\|_{L^\infty(-1,1)}.$}
\end{lem}
\begin{proof}
Note that $q\leq p$ and $\f1q+\f1p>\f12$,  we deduce by  the product
laws in Besov space that \beno
\begin{split}
\normBpo{\mu a\tri w -a\na p}{-1+\f2p}\lesssim&\mu\int_0^t\normbq{a}{\f2q}(\normbp{w}{1+\f2p}+\f{1}{\mu}\normbp{\na p}{-1+\f2p})\,d\tau\\
\lesssim&\mu\|a\|_{L^1_{t,f_1}(B^{\f2q}_{q,1})}.
\end{split}
\eeno Along the same line, we get that \beno
\begin{split}
\normbp{(1+a)\dive[(\tilde{\mu}(a)-\mu)\cM(w)]}{-1+\f2p}&\lesssim (1+\normbq{a}{\f2q})\normbp{(\tilde{\mu}(a)-\mu)\cM(w)}{\f2p}\\
&\lesssim (1+\normbq{a}{\f2q})\normbq{\tilde{\mu}(a)-\mu}{\f2q}\normbp{\cM(w)}{\f2p}\\
&\leq
\frak{C}(1+\normbq{a}{\f2q})\normbq{a}{\f2q}\normbp{w}{1+\f2p},
\end{split}
\eeno for some positive constant $\frak{C}$ depending on
$\|\wt{\mu}'\|_{L^\infty(-1,1)}$ as long as $\|a\|_{L^\infty}\leq
1.$ This gives rise to \beno
\begin{split}
\normBpo{(1+a)\dive[(\tilde{\mu}(a)-\mu)\cM(w)]}{-1+\f2p}& \leq
\frak{C}(1+\normBq{a}{\f2q})\int_0^t\normbq{a}{\f2q}\normbp{w}{1+\f2p}d\tau\\
&\leq \frak{C}
(1+\normBq{a}{\f2q})\|a\|_{L^1_{t,f_1}(B^{\f2q}_{q,1})}.
\end{split}
\eeno This finishes the proof of Lemma \ref{presslem2}.
\end{proof}

Now let us turn to the proof of Proposition \ref{pressprop}.\\

\noindent\textbf{Proof of Proposition \ref{pressprop}.} As both the
existence and uniqueness parts of Proposition \ref{pressprop}
basically follows from the uniform estimate \eqref{pressest} for
appropriate approximate solutions of  \eqref{presseq}. For
simplicity, we just prove \eqref{pressest} for smooth enough
solutions of \eqref{presseq}. Indeed thanks to \eqref{presseq}, we
have \beno
\begin{split}
\na \Pi_{\bar{\la}} = \na (-\tri)^{-1}\Bigl(\dive(a\na \Pi_{\bar{\la}}) + \dive F_{\bar{\la}}
- \dive(v\cdot \na v_{\bar{\la}} + v_{\bar{\la}}\cdot \na w + w\cdot \na v_{\bar{\la}})\\
+\dive \big((1+a)\dive((\wt{\mu}(a)-\mu)\cM(v_{\bar{\la}}))\big) +
\mu\dive(a\tri v_{\bar{\la}})\Bigr),
\end{split}
\eeno from which, we deduce that \beq\label{2.9}
\begin{split}
\normBpo{\na \Pi_{\bar{\la}}}{-1+\f2p} \leq & C\Bigl\{\normBpo{a\na \Pi_{\bar{\la}}}{-1+\f2p}+\normBpo{F_{\bar{\la}}}{-1+\f2p}
+\normBpo{v\cdot \na v_{\bar{\la}}}{-1+\f2p}\\
&+\normBpo{v_{\bar{\la}}\cdot \na w + w\cdot \na v_{\bar{\la}}}{-1+\f2p}+\mu\normBpo{a\tri v_{\bar{\la}}}{-1+\f2p}\\
&+\normBpo{(1+a)\dive((\wt{\mu}(a)-\mu)\cM(v_{\bar{\la}}))}{-1+\f2p}\Bigr\}.
\end{split}
\eeq However as $q \leq p$ and $\f1q+\f1p>\f12$, applying standard
product laws in Besov space leads to \beno
\begin{split}
&\normBpo{v\cdot\na v_{\bar{\la}}}{-1+\f2p}\lesssim \normBp{v}{-1+\f2p}\normBpo{v_{\bar{\la}}}{1+\f2p},\\
&\normBpo{a\na \Pi_{\bar{\la}}}{-1+\f2p} \lesssim \normBq{a}{\f2q}\normBpo{\na \Pi_{\bar{\la}}}{-1+\f2p},\\
&\mu\normBpo{a\tri v_{\bar{\la}}}{-1+\f2p} \lesssim \mu\normBq{a}{\f2q}\normBpo{v_{\bar{\la}}}{1+\f2p},\\
&\normBpo{(1+a)\dive((\wt{\mu}(a)-\mu)\cM(v_{\bar{\la}}))}{-1+\f2p}\leq
\frak{C} (1+
\normBq{a}{\f2q})\normBq{a}{\f2q}\normBpo{v_{\bar{\la}}}{1+\f2p},
\end{split}
\eeno for some positive constant $\frak{C}$ depending on
$\|\wt{\mu}'\|_{L^\infty(-1,1)}$ as long as $\|a\|_{L^\infty}\leq
1.$ This along with Lemmas \ref{presslem1} to \ref{presslem2}
implies Proposition \ref{pressprop} provided that $C\normBq{a}{\f2q}
\leq \f12$. \ef

\medskip

\renewcommand{\theequation}{\thesection.\arabic{equation}}
\setcounter{equation}{0}
\section{The global infinite energy solutions to classical 2-D Navier-Stokes system}

In this section, we shall solve the global wellposedness of the
classical Navier-Stokes system \eqref{weq} with initial data $u_0\in
B^{-1+\frac2p}_{p,1}(\R^2)$ for $1< p <4,$ which is not of finite
energy. In general, the global wellposedness to 2-D classical
Navier-Stokes system with initial data in the scaling invariant
Besov spaces and of infinite energy was solved in
\cite{gallplan,ger}. However considering the special structure of $
B^{-1+\frac2p}_{p,1}(\R^2)$ for $1< p <4,$ we shall provide a much
simpler proof than that in \cite{gallplan,ger}, furthermore,
 more detailed information to this solution will be given here. More precisely, we shall split the solution
 $w$ to  \eqref{weq} as
$w_L+\bar{w}$ with $w_L \eqdefa e^{\mu t\tri}u_0.$ Then it follows
from  \eqref{weq} and Lemma \ref{heatkernel} that \beq\label{wLest}
\normBp{w_L}{-1+\f2p} + \mu\normBpo{w_L}{1+\f2p} \leq
C\normbp{u_0}{-1+\f2p}, \eeq and $\bar{w}$ solves \beq\label{wbareq}
\left\{\begin{array}{l}
\pa_t \bar{w} + \bar{w}\cdot \na \bar{w} + \bar{w}\cdot \na w_L + w_L\cdot \na \bar{w}+ w_L \cdot \na w_L- \mu\tri\bar{w}+\na p=0,\\
\dive \bar{w}=0,\\
\bar{w}|_{t=0} = 0.
\end{array}\right.
\eeq

The main result of this section is as follows:

\begin{prop}\label{prop3.1}
{\sl Given solenoidal vector filed  $u_0\in
B^{-1+\frac2p}_{p,1}(\R^2)$ for $p\in (1,4),$ \eqref{weq} has a
unique solution $w$ of the form: $w_L+\bar{w},$ with $\bar{w}\in
\cC([0,\infty);B^{0}_{2,1}(\R^2))\cap \wt{L}^\infty(\R^+;
B^{0}_{2,1}(\R^2))\cap L^1(\R^+;B^{2}_{2,1}(\R^2)),$ and there holds
\beq\label{west}
\begin{split}
\normBp{w&}{-1+\f2p} +\mu\normBpo{w}{1+\f2p} +\normBpo{\na
p}{-1+\f2p}\\ \leq&
C\normbp{u_0}{-1+\f2p}(1+\normbp{u_0}{-1+\f2p})\exp\Bigl\{\f{C}{\mu^2}\normbp{u_0}{-1+\f2p}^2\Bigr\}.
\end{split}
\eeq}
\end{prop}

We start the proof of Proposition \ref{prop3.1} by the following two
technical lemmas.

\begin{lem}\label{wbarlem1}
{\sl Let $p\in [1,\infty]$, $w_L\in
\wt{L}_t^\infty(B^{-1+\f2p}_{p,1})\cap L^1_t(B^{1+\f2p}_{p,1})$
 and $\bar{w}\in \wt{L}_t^\infty(B^0_{2,1})\cap L^1_t(B^2_{2,1})$ be divergence free vector fields, then for any $\ep>0$, one has
\beno \normBo{\bar{w}\cdot \na w_L + w_L\cdot \na \bar{w}}{0}
&\lesssim& \ep \normBo{\bar{w}}{2} +
\int_0^t\bigl(\normbp{w_L}{1+\f2p}+\f{1}{\ep}\normbp{w_L}{\f2p}^2\bigr)\normb{\bar{w}}{0}\,d\tau.
\eeno}
\end{lem}
\begin{proof} The proof of this lemma basically follows from that of
Lemma \ref{presslem1}. Note that $\dive \bar{w}=\dive w_L =0$, we
have \beno \normBo{\bar{w}\cdot \na w_L + w_L\cdot \na
\bar{w}}{0}\lesssim \normBo{\bar{w}w_L}{1}, \eeno and we get by
applying Bony's decomposition \beno \bar{w}w_L= T_{w_L}\bar{w} +
T_{\bar{w}}w_L + R(\bar{w},w_L). \eeno Applying Lemma \ref{bernlem}
yields \beno
\begin{split}
\normo{\D_j (T_{\bar{w}}w_L)} &\lesssim \int_0^t \sum_{|k-j|\leq 5}\normld{S_{k-1}\bar{w}}\normsup{\D_k w_L}\,d\tau\\
&\lesssim d_j 2^{-j}
\int_0^t\normbp{w_L}{1+\f2p}\normb{\bar{w}}{0}\,d\tau ,
\end{split}
\eeno and \beno
\begin{split}
\normo{\D_j (R(\bar{w},w_L))} &\lesssim \int_0^t 2^{j\f2p}\sum_{k\geq j-N_0}\normld{\wt{\D}_k \bar{w}}\normp{\D_k w_L}d\tau\\
&\lesssim \int_0^t 2^{j\f2p}\sum_{k\geq j-N_0}d_k(\tau)
2^{-k(1+\f2p)}\normbp{w_L}{1+\f2p}\normb{\bar{w}}{0}d\tau\\
& \lesssim d_j 2^{-j}\int_0^t
\normbp{w_L}{1+\f2p}\normb{\bar{w}}{0}d\tau.
\end{split}
\eeno Along the same line, one has \beno
\begin{split} \normo{\D_j (T_{w_L}\bar{w})}
&\lesssim \int_0^t \sum_{|k-j|\leq 5}\normsup{S_{k-1}w_L}\normld{\D_k \bar{w}}d\tau \lesssim \sum_{|k-j|\leq 5}
\int_0^t \normbp{w_L}{\f2p}\normld{\D_k \bar{w}}\,d\tau\\
&\lesssim \Bigl(\sum_{|k-j|\leq 5}\int_0^t2^{-k}\normbp{w_L}{\f2p}^2\normld{\D_k \bar{w}}\,d\tau\Bigr)^{\f12}
\Bigl(\sum_{|k-j|\leq 5}\int_0^t2^k\normld{\D_k \bar{w}}\,d\tau\Bigr)^{\f12}\\
&\lesssim  d_j 2^{-j}
\bigl(\f{1}{\ep}\int_0^t\normbp{w}{\f2p}^2\normb{\bar{w}}{0}\,d\tau
+ \ep \int_0^t \normb{\bar{w}}{2}d\tau\bigr).
\end{split}
\eeno By summing up the above estimates, we finish the  proof of
Lemma \ref{wbarlem1}.
\end{proof}

\begin{lem}\label{wbarlem2}
{\sl Let $p\in (1,4)$ and  $w_L$ be given  at the beginning of this
section,  one has \beq\label{wbarlem2af} \normBo{w_L\cdot \na
w_L}{0}\lesssim \f1{\mu}\|u_0\|_{B^{-1+\f2p}_{p,1}}^2. \eeq}
\end{lem}
\begin{proof} Indeed due to $\dv w_L=0,$ one has
\beno \normBo{w_L\cdot \na w_L}{0}\lesssim \normBo{w_L\otimes
w_L}{1}, \eeno and thanks to Bony's decomposition, we get \beno
w_L\otimes w_L= 2T_{w_L} w_L + R(w_L, w_L). \eeno We first deal with
\eqref{wbarlem2af} for the case when $2\leq p<4.$ In this case,
 we have $p<\f{2p}{p-2}\leq \infty,$ so  that applying Lemma
\ref{bernlem} gives \beno
\begin{split}
\|S_{k-1}w_L\|_{L^\infty_t(L^{\f{2p}{p-2}})}\lesssim &\sum_{\ell\leq
k-2} 2^{2\ell(\f1p-\f{p-2}{2p})}\|\D_\ell w_L\|_{L^\infty_t(L^p)}\\
\lesssim & d_k
2^{\f{2k}p}\|w_L\|_{\wt{L}^\infty_t(B^{-1+\f2p}_{p,1})}\lesssim d_k
2^{\f{2k}p}\|u_0\|_{B^{-1+\f2p}_{p,1}},
\end{split}
\eeno where we used \eqref{wLest} in the last step. Then applying
Lemma \ref{bernlem} once again leads to \beq\label{wbarlem2ad}
\begin{split}
\normo{\D_j (T_{w_L} w_L)} &\lesssim  \sum_{|k-j|\leq 5}\|S_{k-1}w_L\|_{L^\infty_t(L^{\f{2p}{p-2}})}\|\D_k  w_L\|_{L^1_t(L^p)}\\
&\lesssim \f1{\mu}d_j2^{-j}\|u_0\|_{B^{-1+\f2p}_{p,1}}^2.
\end{split}
\eeq Similarly as $2\leq p<4,$ we get by  applying Lemma
\ref{bernlem} that \beno
\begin{split}
\normo{\D_j(R(w_L, w_L))} \lesssim&  2^{2j(\f2p-\f12)}\sum_{k\geq j-N_0}\|\wt{\D}_k w_L\|_{L^\infty_t(L^p)}\|\D_k  w_L\|_{L^1_t(L^p)}\\
&\lesssim  \f1{\mu} 2^{j(\f4p-1)}\sum_{k\geq j-N_0}d_k 2^{-k\f4p}\|u_0\|_{B^{-1+\f2p}_{p,1}}^2 \\
&\lesssim \f1{\mu}d_j2^{-j}\|u_0\|_{B^{-1+\f2p}_{p,1}}^2.
\end{split}
\eeno  This together with \eqref{wbarlem2ad} proves
\eqref{wbarlem2af} for $p\in [2,4).$

On the other hand, when $p\in (1,2),$ let $p'$ be determined by
$\f1{p'}=1-\f1p,$ we deduce from Lemma \ref{bernlem} that
 \beno
\begin{split}
\|S_{k-1}w_L\|_{L^\infty_t(L^{p'})}\lesssim &\sum_{\ell\leq
k-2} 2^{2\ell(\f2p-1)}\|\D_\ell w_L\|_{L^\infty_t(L^p)}\\
\lesssim & d_k
2^{k(\f{2}p-1)}\|w_L\|_{\wt{L}^\infty_t(B^{-1+\f2p}_{p,1})}\lesssim
d_k 2^{k(\f{2}p-1)}\|u_0\|_{B^{-1+\f2p}_{p,1}},
\end{split}
\eeno and \beq\label{wbarlem2ap}
\begin{split}
\normo{\D_j (T_{w_L} w_L)} &\lesssim  2^j\sum_{|k-j|\leq 5}\|S_{k-1}w_L\|_{L^\infty_t(L^{p'})}\|\D_k  w_L\|_{L^1_t(L^p)}\\
&\lesssim \f1{\mu}d_j2^{-j}\|u_0\|_{B^{-1+\f2p}_{p,1}}^2.
\end{split}
\eeq Along the same line, one has
 \beno
\begin{split}
\normo{\D_j(R(w_L, w_L))} \lesssim&  2^{j}\sum_{k\geq j-N_0}\|\wt{\D}_k w_L\|_{L^\infty_t(L^{p'})}\|\D_k  w_L\|_{L^1_t(L^p)}\\
&\lesssim  \f1{\mu} 2^{j}\sum_{k\geq j-N_0}d_k 2^{-2k}\|u_0\|_{B^{-1+\f2p}_{p,1}}^2 \\
&\lesssim \f1{\mu}d_j2^{-j}\|u_0\|_{B^{-1+\f2p}_{p,1}}^2.
\end{split}
\eeno  This together with \eqref{wbarlem2ap} ensures
\eqref{wbarlem2af} for $p\in (1,2).$
\end{proof}

With the above two technical lemmas, we now present the proof of
Proposition \ref{prop3.1}.\\

\no{\bf Proof of Proposition \ref{prop3.1}.}\ As the existence part
of Proposition \ref{prop3.1} essentially follows from \eqref{west}.
Again for simplicity, we just present the detailed proof to
\eqref{west} for smooth enough solutions of \eqref{weq}.  We first
get by taking the $L^2$ inner product of \eqref{wLest} with
$\bar{w}$ that \beno
\begin{split}
\f12\f{d}{dt}\normld{\bar{w}(t)}^2 + \mu\normld{\na \bar{w}(t)}^2
\leq& \normld{\bar{w}}^2\normsup{\na w_L} +
\normld{\bar{w}}\normld{w_L\cdot\na w_L}. \end{split} \eeno Applying
Gronwall's inequality and then using   \eqref{wLest}, Lemma
\ref{wbarlem2}, we get that \beq\label{wbarenergy1}
\begin{split}
\|\bar{w}\|_{L^\infty_t(L^2)}\leq& \normo{w_L\cdot\na
w_L}\exp\bigl\{\|\na w_L\|_{L^1_t(L^\infty)}\bigr\}\\
\leq&\f{C}{\mu}\normbp{u_0}{-1+\f2p}^2\exp\Bigl\{\f{C}{\mu}\normbp{u_0}{-1+\f2p}\Bigr\}\leq
C\normbp{u_0}{-1+\f2p}\exp\Bigl\{\f{C}{\mu}\normbp{u_0}{-1+\f2p}\Bigr\},
\end{split}
\eeq and \beq\label{wbarenergy}
\begin{split}
 \mu\|\na \bar{w}\|_{L^2_t(L^2)}^2
&\leq \|\bar{w}\|_{L^\infty_t(L^2)}^2\|\na w_L\|_{L^1_t(L^\infty)}+\|\bar{w}\|_{L^\infty_t(L^2)}\|w_L\cdot\na w_L\|_{L^1_t(L^2)}\\
&\leq\f{C}{\mu}\normbp{u_0}{-1+\f2p}^3\exp\Bigl\{\f{C}{\mu}\normbp{u_0}{-1+\f2p}\Bigr\}\\
&\leq
C\normbp{u_0}{-1+\f2p}^2\exp\Bigl\{\f{C}{\mu}\normbp{u_0}{-1+\f2p}\Bigr\}.
\end{split}
\eeq

On the other hand, we notice that \beno
\begin{split}
\normBo{\bar{w}\cdot\na\bar{w}}{0}\leq& C\int_0^t\normb{\bar{w}\cdot\na\bar{w}}{0}\,d\tau
\leq C\int_0^t \|\bar{w}\|_{\dot{H}^{\f12}}\|\na\bar{w}\|_{\dot{H}^{\f12}}\,d\tau\\
\leq& C\int_0^t
\normld{\bar{w}}^{\f12}\normld{\na\bar{w}}\normb{\bar{w}}{2}^{\f12}\,d\tau.
\end{split}
\eeno So that it follows from \eqref{wbareq}, Proposition
\ref{stokesprop} and Lemma \ref{wbarlem1} that \beno
\begin{split}
&\normB{\bar{w}}{0} + \mu\normBo{\bar{w}}{2} +\normBo{\na p}{0}\\
&\leq C\Bigl\{\normBo{\bar{w}\cdot\na\bar{w}}{0} + \normBo{\bar{w}\cdot\na w_L +w_L\cdot\na\bar{w}}{0}+\normBo{w_L\cdot\na w_L}{0}\Bigr\}\\
&\leq C\Bigl\{\ep\normBo{\bar{w}}{2} + \|\bar{w}\|_{L^\infty_t(L^2)}\|\na\bar{w}\|_{L^2_t(L^2)}^2 \\
&\qquad\quad+\int_0^t\bigl(\normbp{w_L}{1+\f2p}+\f{1}{\ep}\normbp{w_L}{\f2p}^2\bigr)\normb{\bar{w}}{0}\,d\tau
+\f{C}{\mu}\normbp{u_0}{-1+\f2p}^2\Bigr\}.
\end{split}\eeno Taking $\ep = \f{\mu}{2C}$ in the above inequality and using
\eqref{wLest}, \eqref{wbarenergy},  we infer \beq\label{wbarest}
\begin{split}
&\normB{\bar{w}}{0} + \mu\normBo{\bar{w}}{2} +\normBo{\na p}{0}\\
&\leq
C\exp\Bigl\{C\int_0^t(\normbp{w_L}{1+\f2p}+\f{1}{\mu}\normbp{w_L}{\f2p}^2)\,d\tau\Bigr\}\\
&\qquad\times\Bigl(\f{C}{\mu}\normbp{u_0}{-1+\f2p}^3\exp\bigl\{\f{C}{\mu}\normbp{u_0}{-1+\f2p}\bigr\}+\f{C}{\mu}\normbp{u_0}{-1+\f2p}^2\Bigr)\\
&\leq
C\normbp{u_0}{-1+\f2p}(1+\normbp{u_0}{-1+\f2p})\exp\Bigl\{\f{C}{\mu^2}\normbp{u_0}{-1+\f2p}^2\Bigr\}.
\end{split}
\eeq

Therefore, summing up \eqref{wLest} and \eqref{wbarest} results in
\beno
\begin{split}
\normBp{w}{-1+\f2p}& +\mu\normBpo{w}{1+\f2p} +\normBpo{\na p}{-1+\f2p}\\
\leq &\bigl(\normBp{w_L}{-1+\f2p} +\mu\normBpo{w_L}{1+\f2p}\bigr) \\
&+ \bigl(\normBp{\bar{w}}{-1+\f2p} +\mu\normBpo{\bar{w}}{1+\f2p} +\normBpo{\na p}{-1+\f2p}\bigr)\\
\leq&C\normbp{u_0}{-1+\f2p}(1+\normbp{u_0}{-1+\f2p})\exp\Bigl\{\f{C}{\mu^2}\normbp{u_0}{-1+\f2p}^2\Bigr\},
\end{split}
\eeno which gives rise to \eqref{west}. The uniqueness part of
Proposition \ref{prop3.1} has been proved in  \cite{gallplan,ger}.
This completes the proof of the proposition. \ef

\medskip

\renewcommand{\theequation}{\thesection.\arabic{equation}}
\setcounter{equation}{0}
\section{The proof of Theorem \ref{mainthm}}

The goal of this section is to present the proof of Theorem
\ref{mainthm}. In fact, given $a_0\in B^{\f2q}_{q,1}(\R^2)$, $u_0\in
B^{-1+\f2p}_{p,1}(\R^2)$ with $\normbq{a_0}{\f2q}$ being
sufficiently small and $p,q$ satisfying the conditions listed in
Theorem \ref{mainthm}, it follows by a similar argument as that in
\cite{abipai} that there exists a positive time $T$ so that
\eqref{INS1} has a unique solution $(a,u,\na \Pi)$ with
\begin{equation}\label{local}
\begin{split}
&a\in\cC([0,T];B^{\f2q}_{q,1}(\mathbb{R}^2)),\qquad u\in\cC([0,T];B^{-1+\f2p}_{p,1}(\mathbb{R}^2))\cap L^1((0,T);B^{1+\f2p}_{p,1}(\mathbb{R}^2)),\\
&\na\Pi\in L^1((0,T);B^{-1+\f2p}_{p,1}(\mathbb{R}^2)).
\end{split}
\end{equation} Moreover, if $\f1p+\f1q\geq 1,$
this solution is unique. We denote $T^*$ to be the largest possible
time so that there holds \eqref{local}. Hence the proof of Theorem
\ref{mainthm} is reduced to show  that $T^*=\infty$ under the
assumption of \eqref{thmassume}. Toward this, we split the velocity
$u$ as $w+v,$ with $(w,p), (a,v, \Pi_1)$ solving \eqref{weq} and
\eqref{veq} respectively. Then thanks to Proposition \ref{prop3.1},
it remains to solve \eqref{veq} globally.

\subsection{The estimate of $v$.} First we reformulate the $v$ equation
of  \eqref{veq} to be
\begin{equation}\label{veq1}
\begin{split}
\pa_t v - \mu\tri v =& F -(1+a)\na \Pi_1 + \mu a\tri v + (1+a) \dive[(\wt{\mu}(a)-\mu)\cM(v)]\\
&-(v\cdot \na v + v\cdot \na w + w\cdot \na v).
\end{split}
\end{equation}
Let $f_1(t)$, $f_2(t)$, $a_{\bar{\la}}$, $v_{\bar{\la}}$, $\na
\Pi_{\bar{\la}}$ be given by  \eqref{prop2.3}, and $a_{\la_1}\eqdefa
a\exp\bigl\{-\la_1\int_0^tf_1(\tau)d\tau\bigr\}$. Then  it follows
from  \eqref{veq1} that \beno
\begin{split}
\pa_t& v_{\bar{\la}}+ (\la_1f_1(t)+\la_2f_2(t))v_{\bar{\la}}-\mu\tri v_{\bar{\la}}= F_{\bar{\la}} -(1+a)\na\Pi_{\bar{\la}}+\mu a\tri v_{\bar{\la}} \\
&+(1+a) \dive[(\wt{\mu}(a)-\mu)\cM(v_{\bar{\la}})]-(v\cdot \na
v_{\bar{\la}} + v_{\bar{\la}}\cdot \na w + w\cdot \na
v_{\bar{\la}}).
\end{split}\eeno
Applying $\D_j$ to the above equation and taking the $L^2$ inner
product of the resulting equation with $|\D_j
v_{\bar{\la}}|^{p-2}\D_j v_{\bar{\la}} $ (in the case when $p\in
(1,2),$ we need a small modification to make this argument rigorous,
which we omit here), we obtain \beq\label{veqaf}
\begin{split}
\f1p\f{d}{dt}& \|\D_jv_{\bar{\la}}\|_{L^p}^p+
(\la_1f_1(t)+\la_2f_2(t))\|\D_jv_{\bar{\la}}\|_{L^p}^p
-\mu\bigl(\tri\D_j v_{\bar{\la}}\ |\ |\D_j v_{\bar{\la}}|^{p-2}\D_j
v_{\bar{\la}}\bigr)\\
\leq &\Bigl\{\|\D_j F_{\bar{\la}}\|_{L^p}
+\|\D_j((1+a)\na\Pi_{\bar{\la}})\|_{L^p}+\|\D_j((1+a)
\dive[(\wt{\mu}(a)-\mu)\cM(v_{\bar{\la}})])\|_{L^p}\\
&+\mu \|\D_j(a\tri v_{\bar{\la}})\|_{L^p} +\|\D_j(v\cdot \na
v_{\bar{\la}} + v_{\bar{\la}}\cdot \na w + w\cdot \na
v_{\bar{\la}})\|_{L^p}\Bigr\}\|\D_jv_{\bar{\la}}\|_{L^p}^{p-1}.
\end{split}\eeq
While applying Lemma A.5 of \cite{danchin01} that \beno
-\bigl(\tri\D_j v_{\bar{\la}}\ |\ |\D_j v_{\bar{\la}}|^{p-2}\D_j
v_{\bar{\la}}\bigr)\geq \bar{c}2^{2j}\|\D_jv_{\bar{\la}}\|_{L^p}^p
\eeno for some positive constant $\bar{c},$ from which and
\eqref{veqaf}, we deduce  that \beno
\begin{split}\|\D_j&v_{\bar{\la}}\|_{L^\infty_t(L^p)}+
\int_0^t(\la_1f_1(t')+\la_2f_2(t'))\|\D_jv_{\bar{\la}}\|_{L^p}\,dt'
+\bar{c}\mu\|\D_jv_{\bar{\la}}\|_{L^1_t(L^p)}\\
\leq &\Bigl\{\|\D_j F_{\bar{\la}}\|_{L^1_t(L^p)}
+\|\D_j((1+a)\na\Pi_{\bar{\la}})\|_{L^1_t(L^p)}+\|\D_j((1+a)
\dive[(\wt{\mu}(a)-\mu)\cM(v_{\bar{\la}})])\|_{L^1_t(L^p)}\\
&+\mu \|\D_j(a\tri v_{\bar{\la}})\|_{L^1_t(L^p)} +\|\D_j(v\cdot \na
v_{\bar{\la}} + v_{\bar{\la}}\cdot \na w + w\cdot \na
v_{\bar{\la}})\|_{L^1_t(L^p)}\Bigr\}.
\end{split}
\eeno This gives rise to \beq\label{veqadf}
\begin{split}
&\normBp{v_{\bar{\la}}}{-1+\f2p} +\la_1\|v_{\bar{\la}}\|_{L^1_{t,f_1}(B^{-1+\f2p}_{p,1})}
+\la_2\|v_{\bar{\la}}\|_{L^1_{t,f_2}(B^{-1+\f2p}_{p,1})} +\bar{c}\mu\normBpo{v_{\bar{\la}}}{1+\f2p}\\
&\leq C\Bigl\{\normBpo{F_{\bar{\la}}}{-1+\f2p} +\normBpo{(1+a)\na\Pi_{\bar{\la}}}{-1+\f2p} +\normBp{v}{-1+\f2p}\normBpo{v_{\bar{\la}}}{1+\f2p}\\
&\quad
+\bigl(\mu+\frak{C}(1+\normBq{a}{\f2q})\bigr)\normBq{a}{\f2q}\normBpo{v_{\bar{\la}}}{1+\f2p}+
\normBpo{v_{\bar{\la}}\cdot \na w + w\cdot \na
v_{\bar{\la}}}{-1+\f2p}\Bigr\}, \end{split} \eeq where the norm
$\|v_{\bar{\la}}\|_{L^1_{t,f}(B^{-1+\f2p}_{p,1})}$ is given by
Definition \ref{defpz} and $\frak{C}$ is a positive constant
depending on $\|\wt{\mu}'\|_{L^\infty(-1,1)}$ as long as
$\|a\|_{L^\infty}\leq 1.$

 Let \beq
\label{assum} \bar{T}\eqdefa \sup\bigl\{\ t<T^\ast,\ \
\normBq{a}{\f2q}\leq c_1\ \bigr\} \eeq for some $c_1$ sufficiently
small. Then \eqref{pressest} ensures that for $ t\leq \bar{T}$
\beno\begin{split} \normBpo{(1+a)\na& \Pi_{\bar{\la}}}{-1+\f2p}\\
\leq &
\frak{C} \Bigl\{\ep \normBpo{v_{\bar{\la}}}{1+\f2p} + \normBp{v}{-1+\f2p}\normBpo{v_{\bar{\la}}}{1+\f2p}\\
& +\|v_{\bar{\la}}\|_{L^1_{t,f_1}(B^{-1+\f2p}_{p,1})}+ \f{1}{\ep} \|v_{\bar{\la}}\|_{L^1_{t,f_2}(B^{-1+\f2p}_{p,1})}\\
& + (1+
\mu+\normBq{a}{\f2q})\bigl(\|a_{\bar{\la}}\|_{L^1_{t,f_1}(B^{\f2q}_{q,1})}+\normBq{a}{\f2q}\normBpo{v_{\bar{\la}}}{1+\f2p}\bigr)\Bigr\},
\end{split}
\eeno from which, Lemma \ref{presslem1}, we infer from
\eqref{veqadf} that \beq\label{vest}
\begin{split}
&\normBp{v_{\bar{\la}}}{-1+\f2p} +\la_1\|v_{\bar{\la}}\|_{L^1_{t,f_1}(B^{-1+\f2p}_{p,1})}+\la_2\|v_{\bar{\la}}\|_{L^1_{t,f_2}(B^{-1+\f2p}_{p,1})}
 +\bar{c}\mu\normBpo{v_{\bar{\la}}}{1+\f2p}\\
&\leq \frak{C} \Bigl\{\ep \normBpo{v_{\bar{\la}}}{1+\f2p} +
\normBp{v}{-1+\f2p}\normBpo{v_{\bar{\la}}}{1+\f2p}
+\|v_{\bar{\la}}\|_{L^1_{t,f_1}(B^{-1+\f2p}_{p,1})}+ \f{1}{\ep} \|v_{\bar{\la}}\|_{L^1_{t,f_2}(B^{-1+\f2p}_{p,1})}\\
& \qquad +
(1+\mu+\normBq{a}{\f2q})\bigl(\|a_{\bar{\la}}\|_{L^1_{t,f_1}(B^{\f2q}_{q,1})}+\normBq{a}{\f2q}\normBpo{v_{\bar{\la}}}{1+\f2p}
\bigr)\Bigr\}
\end{split}
\eeq for $t\leq \bar{T}$ and  some positive constant $\frak{C}$
depending on $\|\wt{\mu}'\|_{L^\infty(-1,1)}.$

\subsection{The proof of Theorem \ref{mainthm}.} We get by taking
$\la=\la_1$ in Proposition \ref{aprop} that \beq\label{aest1}
\normBq{a_{\la_1}}{\f2q}+\f{\la_1}{2}\|a_{\la_1}\|_{L^1_{t,f_1}(B^{\f2q}_{q,1})}\leq\normbq{a_0}{\f2q}+C\normBpo{v}{1+\f2p}\normBq{a_{\la_1}}{\f2q}.
\eeq Note that \beno
\|a_{\bar{\la}}\|_{L^1_{t,f_1}(B^{\f2q}_{q,1})}\leq\|a_{\la_1}\|_{L^1_{t,f_1}(B^{\f2q}_{q,1})},
\eeno By summing up \eqref{vest} and \eqref{aest1}$\times (1+\mu)$
and choosing $\e, \la_1,\la_2$ satisfying $\frak{C}\ep =
\f{\bar{c}}{2}\mu$, $\la_1= 8\frak{C}$,
$\la_2=\f{2\frak{C}^2}{\bar{c}\mu},$ we obtain  \beq\label{totalest}
\begin{split}
&(1+\mu)\normBq{a_{\la_1}}{\f2q}+\normBp{v_{\bar{\la}}}{-1+\f2p}+
\f{\la_1}{2}\bigl(\f{1+\mu}{2}\|a_{\la_1}\|_{L^1_{t,f_1}(B^{\f2q}_{q,1})}
+\|v_{\bar{\la}}\|_{L^1_{t,f_1}(B^{-1+\f2p}_{p,1})}\bigr)\\
&\quad+\f{\la_2}{2}\|v_{\bar{\la}}\|_{L^1_{t,f_2}(B^{-1+\f2p}_{p,1})}+\f{\bar{c}\mu}{2}\normBpo{v_{\bar{\la}}}{1+\f2p}\\
&\leq (1+\mu)\normbq{a_0}{\f2q}+C_1\Bigl\{(1+\mu)\normBpo{v}{1+\f2p}\normBq{a_{\la_1}}{\f2q}+\normBp{v}{-1+\f2p}\normBpo{v_{\bar{\la}}}{1+\f2p}\\
&\quad+(1+\mu+\normBq{a}{\f2q})\normBq{a}{\f2q}\normBpo{v_{\bar{\la}}}{1+\f2p}\Bigr\}
\end{split}
\eeq for $t\leq\bar{T}$ and  some positive constant $C_1$ depending
on $\|\wt{\mu}'\|_{L^\infty(-1,1)}.$

Now let $c_2$ be a small enough positive constant, which will be determined later on. We define $\Upsilon$ by
\begin{equation}\label{maxitime}
\Upsilon \overset{def}{=} sup \Bigl\{ t<T^* :
(1+\mu)\normBq{a}{\f2q}+\normBp{v}{-1+\f2p}+\mu\normBpo{v}{1+\f2p}
\leq c_2\mu\ \Bigr\}.
\end{equation}
\eqref{maxitime} together with \eqref{assum} implies that $
\Upsilon\leq\bar{T}$, if we take $c_2\leq c_1$. We shall prove that
$\Upsilon= \infty$ under the assumption of \eqref{thmassume}.
Otherwise, taking $c_2\leq
\min\bigl(\f{\bar{c}}{12C_1},\f{1}{2C_1}\bigr)$, we deduce form
\eqref{totalest} that \beno
\begin{split}
\normBp{v_{\bar{\la}}}{-1+\f2p}+
\f{1+\mu}{2}\normBq{a_{\la_1}}{\f2q}+\f{\bar{c}\mu}4\normBpo{v_{\bar{\la}}}{1+\f2p}
\leq (1+\mu)\normbq{a_0}{\f2q},
\end{split}
\eeno for $t\leq \Upsilon$. This  together with \eqref{prop2.3}
gives rise to \beq\label{last}
\begin{split}
&\normBp{v}{-1+\f2p}+ \f{1+\mu}{2}\normBq{a}{\f2q}+\f{\bar{c}\mu}4\normBpo{v}{1+\f2p} \\
&\leq
(1+\mu)\normbq{a_0}{\f2q}\exp\Bigl\{C_2\int_0^t\bigl(\normbp{w(\tau)}{1+\f2p}+\f{1}{\mu}\normbp{\na
p(\tau)}{-1+\f2p}
+\f{1}{\mu}\normbp{w(\tau)}{\f2p}^2\bigr)\,d\tau\Bigr\},
\end{split}
\eeq
for $t\leq \Upsilon$.

Combining \eqref{last} with \eqref{west}, we reach \beno
\begin{split}
&\normBp{v}{-1+\f2p}+ \f{1+\mu}{2}\normBq{a}{\f2q}+\f{\bar{c}\mu}4\normBpo{v}{1+\f2p}\\
&\leq(1+\mu)\normbq{a_0}{\f2q}\exp\Bigl\{\bigl[\f{C_3}{\mu}\normbp{u_0}{-1+\f2p}(1+\normbp{u_0}{-1+\f2p})\\
&\qquad+\f{C^2_3}{\mu^2}\normbp{u_0}{-1+\f2p}^2(1+\normbp{u_0}{-1+\f2p}^2)\bigr]\exp\bigl(\f{C_3}{\mu^2}
\normbp{u_0}{-1+\f2p}^2\bigr) \Bigr\}\\
&\leq(1+\mu)\normbq{a_0}{\f2q}\exp\Bigl\{C_4(1+\normbp{u_0}{-1+\f2p}^2)\exp\bigl(\f{C_4}{\mu^2}
\normbp{u_0}{-1+\f2p}^2\bigr) \Bigr\}\\
&\leq(1+\mu)\normbq{a_0}{\f2q}\exp\Bigl\{\bar{C}(1+\mu^2)\exp\bigl(\f{\bar{C}}{\mu^2}
\normbp{u_0}{-1+\f2p}^2\bigr) \Bigr\}
 \end{split}\eeno for $t\leq
\Upsilon$ and some positive constants $\bar{C}$ which depends on
 $\|\wt{\mu}'\|_{L^\infty(-1,1)}.$ If we take
$C_0$ large enough and $c_0$ sufficiently small in
\eqref{thmassume}, which depend on $\|\wt{\mu}'\|_{L^\infty(-1,1)},$
there holds \beno
 (1+\mu)\normBq{a}{\f2q}+\normBp{v}{-1+\f2p}+\mu\normBpo{v}{1+\f2p} \leq \f{c_2}{2}\mu \eeno
for $t\leq \Upsilon$, which contradicts with \eqref{maxitime}.
Whence we conclude that $\Upsilon =T^\ast= \infty$. This completes
the proof of  Theorem \ref{mainthm}\ef
\medskip

\renewcommand{\theequation}{\thesection.\arabic{equation}}
\setcounter{equation}{0}
\section{The proof of Theorem \ref{mainthm1}}

The proof of Theorem \ref{mainthm1} basically follows the same line
of the proof to Theorem \ref{mainthm}. More precisely,

\subsection{Estimates of the transport equation}

As we shall not use Lagrange approach  in \cite{dm}, we need first
 to investigate the following transport equation
\beq \label{transeq1} \pa_t f + u \cdot \na f =0, \qquad f|_{t=0} =
f_0 \eeq with initial data $f_0$ in multiplier space of
$B^{s}_{p,1}(\R^2)$ .

\begin{lem}\label{actionest}
{\sl Let $f\in B^{s}_{p,1}(\R^d)$ with $-1<s<1$, and $u\in
L^1((0,T); Lip(\R^d)).$  Let $X_u$ be the flow map determined by
\eqref{flow}. Then $f\circ X_u \in \wt{L}^\infty((0,T);
B^{s}_{p,1}(\R^d))$, and there holds \beq\label{actest} \|f\circ
X_u\|_{\wt{L}^\infty_t(B^{s}_{p,1})}\leq
C\|f\|_{B^{s}_{p,1}}\exp\Bigl\{C\int_0^t\|\na
u(\tau)\|_{L^\infty}\,d\tau\Bigr\}\quad\mbox{for}\quad t\leq T.\eeq}
\end{lem}

\begin{proof} Let $f_\ell \eqdefa \D_\ell f,$  we deduce from Lemma \ref{action} that \beno
\|\D_j(f_\ell\circ X_u)\|_{L^\infty_t(L^p)}\leq Cd_\ell2^{-\ell
s}\|f\|_{B^{s}_{p,1}}\min\bigl(2^{j-\ell},2^{\ell-j}\bigr)\exp\Bigl\{C\int_0^t\|\na
u(\tau)\|_{L^\infty}d\tau\Bigr\},\eeno from which  and  $-1<s<1$, we
infer for any $j\in\Z$ \beq\label{actestb}
\begin{split} \|\D_j(f\circ X_u)\|_{L^\infty_t(L^p)} \leq&
\bigl(\sum\limits_{\ell<j} + \sum\limits_{\ell\geq
j}\bigr)\|\D_j(f_\ell\circ X_u)\|_{L^\infty_t(L^p)}\\
\leq& C\|f\|_{B^{s}_{p,1}}\bigl(\sum\limits_{\ell<j}d_\ell2^{-\ell
s} 2^{\ell-j} + \sum\limits_{\ell\geq
j}d_\ell2^{-\ell s} 2^{j-\ell})\exp\Bigl\{C\int_0^t\|\na u(\tau)\|_{L^\infty}\,d\tau\Bigr\}\\
\leq&Cd_j 2^{-js}\|f\|_{B^{s}_{p,1}}\exp\Bigl\{C\int_0^t\|\na
u(\tau)\|_{L^\infty}\,d\tau\Bigr\},
\end{split}
\eeq this together with Definition \ref{chaleur+} implies
\eqref{actest}, and we complete the proof of the lemma.
\end{proof}

\begin{rmk}\label{rmk5.0} The case when $\f{d}p\geq 1,$  $s\in (-1,\f{d}p),$ and $u\in
L^1((0,T); B^{1+\f{d}p}(\R^d)),$ we have a similar version of Lemma
\ref{actionest}. For simplicity, we just present the case when
$\frac{d}p=1.$ Instead of \eqref{actest}, we shall prove
\beq\label{actesta} \|f\circ
X_u\|_{\wt{L}^\infty_t(B^{1}_{p,1})}\leq
C\|f\|_{B^{1}_{p,1}}(1+\|u\|_{L^1_t(B^{2}_{p,1})})\exp\Bigl\{C\int_0^t\|\na
u(\tau)\|_{L^\infty}\,d\tau\Bigr\}.\eeq We first deduce from
\eqref{actestb} that \beq\label{actestc} \|f\circ
X_u\|_{\wt{L}^\infty_t(B^{1}_{p,\infty})}\leq
C\|f\|_{B^{1}_{p,1}}\exp\Bigl\{C\int_0^t\|\na
u(\tau)\|_{L^\infty}\,d\tau\Bigr\}.\eeq While we get by taking
$\na_y$ to \eqref{flow} that \beno \na_yX_u(t,y) = Id +\int_0^t \na
u(\tau, X_u(\tau,y))\na_yX_u(t,y)\,d\tau,\eeno from which,  and the
standard product laws in Besov spaces, we infer \beno
\begin{split}
\|\na_yX_u-Id\|_{\wt{L}^\infty_t(B^{1}_{p,\infty})}\leq
&C\int_0^t\bigl(\| \na u(\tau,
X_u(\tau,\cdot))\|_{B^1_{p,\infty}}(1+\|\na_yX_u(t,\cdot)\|_{L^\infty})\\
&\qquad+\|\na u(\tau,\cdot)\|_{L^\infty}
\|\na_yX_u-Id\|_{B^{1}_{p,\infty}}\bigr)\,d\tau.
\end{split}
\eeno Applying Gronwall's inequality and \eqref{actestc} gives
\beq\label{actestd}
\|\na_yX_u-Id\|_{\wt{L}^\infty_t(B^{1}_{p,\infty})}\leq
C\|u\|_{L^1_t(B^{2}_{p,1})}\exp\Bigl\{C\int_0^t\|\na
u(\tau)\|_{L^\infty}\,d\tau\Bigr\}. \eeq

On the other hand, notice that \beno \na_y(f\circ X_u)=\na f\circ
X_u(\na_yX_u-Id)+\na f\circ X_u, \eeno from which and Bony's
decomposition, we infer \beq\label{actestdg}
\begin{split}
 \|f\circ&
X_u\|_{\wt{L}^\infty_t(B^{1}_{p,1})}= \|\na_y(f\circ
X_u)\|_{\wt{L}^\infty_t(B^{0}_{p,1})}\\
 \leq& C\|\na f\circ
X_u\|_{\wt{L}^\infty_t(B^0_{p,1})}\bigl(1+\|\na_yX_u-Id\|_{L^\infty_t(L^\infty)}+\|\na_yX_u-Id\|_{\wt{L}^\infty_t(B^{1}_{p,\infty})}\bigr).
\end{split}
\eeq While applying \eqref{actest} yields \beno \|\na f\circ
X_u\|_{\wt{L}^\infty_t(B^0_{p,1})}\leq
C\|f\|_{B^{1}_{p,1}}\exp\Bigl\{C\int_0^t\|\na
u(\tau)\|_{L^\infty}\,d\tau\Bigr\}. \eeno
 This together with  \eqref{actestd} and \eqref{actestdg}  enures
\eqref{actesta}.
\end{rmk}

The main result of this subsection is as follows:

\begin{prop}\label{aprop1}
{\sl Let $2 < p <4$, $-1<s\leq \f2p$ and $u \in
L^1((0,T),B^{1+\f2p}_{p,1}(\R^2))$
  be a divergence free vector fields.
 Then given $f_0 \in
\cM(B^{s}_{p,1}(\R^2))$, \eqref{transeq1} has a unique solution $f
\in L^\infty((0,T); \cM(B^{s}_{p,1}(\R^2))),$ moreover, there holds
\beq \label{fest} \|f\|_{L^\infty_t(\cM(B^{s}_{p,1}))} \leq
C\|f_0\|_{\cM(B^{s}_{p,1})} \exp \Bigl\{C\int_0^t\|\na
u(\tau)\|_{L^\infty}\,d\tau\Bigr\} \eeq for any $t\in(0,T]$. }
\end{prop}

\begin{proof} Both the existence and uniqueness part of Proposition
\ref{aprop1} follows from \eqref{fest}. Indeed let $X_u$ be the flow
map determined by \eqref{flow}. Then we infer from \eqref{transeq1}
that $f(t,x)=f_0(X_u^{-1}(t,x)),$ from which, Definition \ref{deffm}
and Lemma \ref{actionest}, we infer \beno
\begin{split}
\|f(t)\|_{\cM(B^{s}_{p,1})}=&\sup_{\|\psi\|_{B^{s}_{p,1}}=1}\|\psi f(t)\|_{B^{s}_{p,1}}\\
=&\sup_{\|\psi\|_{B^{s}_{p,1}}=1} \|(\psi\circ X_u(t) f_0)\circ X_u^{-1}(t)\|_{B^{s}_{p,1}}\\
\leq& C\sup_{\|\psi\|_{B^{s}_{p,1}}=1}\|\psi\circ
X_u(t)f_0\|_{B^{s}_{p,1}}\exp\Bigl\{C\int_0^t\|\na
u(\tau)\|_{L^\infty}\,d\tau\Bigr\},
\end{split}
\eeno applying  Lemma \ref{actionest} once again leads to \beno
\begin{split}
\|f(t)\|_{\cM(B^{s}_{p,1})}\leq
&C\|f_0\|_{\cM(B^{s}_{p,1})}\exp\Bigl\{C\int_0^t\|\na
u(\tau)\|_{L^\infty}\Bigr\}\sup_{\|\psi\|_{B^{s}_{p,1}}=1}\|\psi\circ
X_u(t)\|_{B^{s}_{p,1}}\\
\leq&C\|f_0\|_{\cM(B^{s}_{p,1})}\exp\Bigl\{C\int_0^t\|\na u(\tau)\|_{L^\infty}\Bigr\}\sup_{\|\psi\|_{B^{s}_{p,1}}=1}\|\psi\|_{B^{s}_{p,1}}\\
\leq&C\|f_0\|_{\cM(B^{s}_{p,1})}\exp\Bigl\{C\int_0^t\|\na
u(\tau)\|_{L^\infty}\Bigr\}\quad\mbox{for any}\quad t\leq T.
\end{split}
\eeno This completes the proof of Proposition \ref{aprop1}.
\end{proof}

\subsection{Estimates of the pressure}
In this subsection, we aim at providing similar version of
Proposition \ref{pressprop} in the case when $a\in L^\infty((0,T);
\cM(B^{-1+\f2p}_{p,1}(\R^2)))$ and $\wt{\mu}(a)-\mu \in
L^\infty((0,T); \cM(B^{\f2p}_{p,1}(\R^2)))$.

\begin{prop}\label{pressprop1}
{\sl Let $p \in [2, 4),$  $a \in L^\infty((0,T);
\cM(B^{-1+\f2p}_{p,1}(\R^2)))$ and $\wt{\mu}(a)-\mu \in
L^\infty((0,T);$ $ \cM(B^{\f2p}_{p,1}(\R^2))).$ Let
 $w,v \in \widetilde{L}_T^\infty(B^{-1+\f2p}_{p,1})\cap L^1_T(B^{1+\f2p}_{p,1})$ and $\na p \in L^1_T(B^{-1+\f2p}_{p,1})$.
 Then \eqref{presseq} has
a unique solution with $\na \Pi_1 \in L^1_T(B^{-1+\f2p}_{p,1}),$ and
for any $ \ep >0,$ $t\leq T,$ there holds \beq \label{pressest1}
\begin{split}
\normBpo{\na \Pi_{\bar{\la}}}{-1+\f2p} &\leq
\f{C}{1-C\|a\|_{L^\infty_T(\cM(B^{-1+\f2p}_{p,1}))}} \Bigl\{\ep
\normBpo{v_{\bar{\la}}}{1+\f2p} +
 \normBp{v}{-1+\f2p}\normBpo{v_{\bar{\la}}}{1+\f2p}\\
& +\|v_{\bar{\la}}\|_{L^1_{t,f_1}(B^{-1+\f2p}_{p,1})}+ \f{1}{\ep}
\|v_{\bar{\la}}\|_{L^1_{t,f_2}(B^{-1+\f2p}_{p,1})} \\
&+\bigl(\normBpo{v_{\bar{\la}}}{1+\f2p}+\int_0^tf_1(\tau)d\tau\bigr)\bigl[\mu
\|a\|_{L^\infty_t(
\cM(B^{-1+\f2p}_{p,1}))}\\
&+(1+\|a\|_{L^\infty_t( \cM(B^{-1+\f2p}_{p,1}))})\|\wt{\mu}(a)-\mu
\|_{ L^\infty_t(\cM(B^{\f2p}_{p,1}))}\bigr]\Bigr\}
\end{split}
\eeq provided that $C\|a\|_{L^\infty_T(\cM(B^{-1+\f2p}_{p,1}))} \leq
\f12,$ where $f_1(t), f_2(t)$  and $\Pi_{\bar{\la}}, v_{\bar{\la}}$
are defined by \eqref{prop2.3}.}
\end{prop}
\begin{proof}  Similar to the proof of Proposition \ref{pressprop}, we just present the proof of \eqref{pressest1}
for smooth enough solutions of \eqref{presseq}. Indeed along the
same line to the proof of Proposition \ref{pressprop}, we have
\eqref{2.9}. While applying Definition \ref{deffm} and standard
product laws in Besov spaces leads to \beno
\begin{split}
&\normBpo{v\cdot\na v_{\bar{\la}}}{-1+\f2p}\lesssim \normBp{v}{-1+\f2p}\normBpo{v_{\bar{\la}}}{1+\f2p},\\
&\normBpo{a\na \Pi_{\bar{\la}}}{-1+\f2p} \lesssim \|a\|_{L^\infty_t(\cM(B^{-1+\f2p}_{p,1}))}\normBpo{\na \Pi_{\bar{\la}}}{-1+\f2p},\\
&\normBpo{a\tri v_{\bar{\la}}}{-1+\f2p} \lesssim \|a\|_{L^\infty_t( \cM(B^{-1+\f2p}_{p,1}))}\normBpo{v_{\bar{\la}}}{1+\f2p},\\
&\normBpo{a(\mu\tri w-\na p)}{-1+\f2p}\lesssim
\mu\|a\|_{L^\infty_t(\cM(B^{-1+\f2p}_{p,1}))}\int_0^t
f_1(\tau)\,d\tau,
\end{split}
\eeno and \beno
\begin{split}
&\normBpo{(1+a)\dive((\wt{\mu}(a)-\mu)\cM(v_{\bar{\la}}))}{-1+\f2p}\\
&\qquad\lesssim (1+\|a\|_{L^\infty_t(
\cM(B^{-1+\f2p}_{p,1}))})\|\wt{\mu}(a)-\mu
\|_{ L^\infty_t(\cM(B^{\f2p}_{p,1}))}\normBpo{v_{\bar{\la}}}{1+\f2p},\\
&\normBpo{(1+a)\dive((\wt{\mu}(a)-\mu)\cM(w))}{-1+\f2p}\\
&\qquad\lesssim (1+\|a\|_{L^\infty_t(
\cM(B^{-1+\f2p}_{p,1}))})\|\wt{\mu}(a)-\mu \|_{
L^\infty_t(\cM(B^{\f2p}_{p,1}))}\int_0^t f_1(\tau)\,d\tau,
\end{split}
\eeno so that \beq\label{fla}
\begin{split}
\|F_{\bar{\la}}&\|_{L^1_t(B^{-1+\f2p}_{p,1})}\leq \normBpo{a(\mu\tri
w-\na
p)}{-1+\f2p}+\normBpo{(1+a)\dive((\wt{\mu}(a)-\mu)\cM(w))}{-1+\f2p}\\
\lesssim& \bigl[\mu\|a\|_{L^\infty_t(
\cM(B^{-1+\f2p}_{p,1}))}+(1+\|a\|_{L^\infty_t(
\cM(B^{-1+\f2p}_{p,1}))})\|\wt{\mu}(a)-\mu \|_{
L^\infty_t(\cM(B^{\f2p}_{p,1}))}\bigr]\int_0^t f_1(\tau)\,d\tau.
\end{split}
\eeq Substituting the above estimates and Lemma \ref{presslem1} into
\eqref{2.9} ensures \eqref{pressest1} provided that $$
C\|a\|_{L^\infty_T(\cM(B^{-1+\f2p}_{p,1}))}\leq \f12$$ This
completes the proof of  Proposition \ref{pressprop1}.
\end{proof}

\subsection{The proof of Theorem
\ref{mainthm1}}
 For $p\in (2,4),$ given
$a_0\in \cM(B^{-1+\f2p}_{p,1}(\R^2))$ with $\wt{\mu}(a_0)-\mu \in
\cM(B^{\f2p}_{p,1}(\R^2))$, $u_0\in B^{-1+\f2p}_{p,1}(\R^2)$ with
$\|a_0\|_{\cM(B^{-1+\f2p}_{p,1})}+\|\wt{\mu}(a_0)-\mu\|_{\cM(B^{\f2p}_{p,1})}$
being sufficiently small, it follows from Theorem 2 in \cite{dm} and
Proposition \ref{aprop1} that there exists a positive time $T$ so
that \eqref{INS1} has a unique solution $(a,u,\na \Pi)$ with
\begin{equation}\label{local1}
\begin{split}
&a\in L^\infty((0,T);\cM(B^{-1+\f2p}_{p,1}(\R^2))),\qquad
\wt{\mu}(a)-\mu \in
L^\infty((0,T);\cM(B^{\f2p}_{p,1}(\R^2))),\\
&u\in\cC([0,T];B^{-1+\f2p}_{p,1}(\mathbb{R}^2))\cap
L^1((0,T);B^{1+\f2p}_{p,1}(\mathbb{R}^2)),\qquad \na\Pi\in
L^1((0,T);B^{-1+\f2p}_{p,1}(\mathbb{R}^2)).
\end{split}
\end{equation}
We denote by $T^*$  the largest possible time so that there holds
\eqref{local1}. Hence the proof of Theorem \ref{mainthm1} is reduced
to show  that $T^*=\infty$ provided that there holds
\eqref{thmassumea}. Toward this, as in the proof of Theorem
\ref{mainthm},  we split the velocity field $u$ as $w+v,$ with $w,
(a,v)$ solving \eqref{weq} and \eqref{veq} respectively. Then thanks
to Proposition \ref{prop3.1}, it remains to solve \eqref{veq}
globally. In order to do so, let $f_1(t)$, $f_2(t)$,
$v_{\bar{\la}}$, $\na \Pi_{\bar{\la}}$ be given by \eqref{prop2.3},
along the same line to the proof of Theorem \ref{mainthm}, we deduce
from \eqref{veqaf} that \beq\label{locala}
\begin{split}
&\normBp{v_{\bar{\la}}}{-1+\f2p} +\la_1\|v_{\bar{\la}}\|_{L^1_{t,f_1}(B^{-1+\f2p}_{p,1})}+\la_2\|v_{\bar{\la}}\|_{L^1_{t,f_2}(B^{-1+\f2p}_{p,1})} +\bar{c}\mu\normBpo{v_{\bar{\la}}}{1+\f2p}\\
&\leq
C\Bigl\{\|F_{\bar{\la}}\|_{L^1_t(B^{-1+\f2p}_{p,1})}+\normBpo{(1+a)\na\Pi_{\bar{\la}}}{-1+\f2p}
+\normBp{v}{-1+\f2p}\normBpo{v_{\bar{\la}}}{1+\f2p}\\
&\qquad+\normBpo{v_{\bar{\la}}\cdot \na w + w\cdot \na
v_{\bar{\la}}}{-1+\f2p}+\mu\|a\|_{L^\infty_t(\cM(B^{-1+\f2p}_{p,1}))}\normBpo{v_{\bar{\la}}}{1+\f2p}\\
&\qquad
+\normBpo{(1+a)\dive((\wt{\mu}(a)-\mu)\cM(v_{\bar{\la}}))}{-1+\f2p}\Bigr\},
\end{split} \eeq where the norm $\|v_{\bar{\la}}\|_{L^1_{t,f}(B^{-1+\f2p}_{p,1})}$ is given by Definition \ref{defpz}.

  We denote \beq \label{assuma} \bar{T}\eqdefa
\sup\bigl\{\ t<T^\ast,\ \ \mu\|a\|_{L^\infty_t(
\cM(B^{-1+\f2p}_{p,1}))}+\|\wt{\mu}(a)-\mu \|_{
L^\infty_t(\cM(B^{\f2p}_{p,1}))}\leq c_1\mu\ \bigr\} \eeq for some
$c_1$ being sufficiently small. Then we get by substituting
\eqref{pressest1} and \eqref{fla} into \eqref{locala} that for $
t\leq \bar{T}$ \beq\label{vesta}
\begin{split}
&\normBp{v_{\bar{\la}}}{-1+\f2p} +\la_1\|v_{\bar{\la}}\|_{L^1_{t,f_1}(B^{-1+\f2p}_{p,1})}+\la_2\|v_{\bar{\la}}\|_{L^1_{t,f_2}(B^{-1+\f2p}_{p,1})} +\bar{c}\mu\normBpo{v_{\bar{\la}}}{1+\f2p}\\
&\leq C \Bigl\{\ep \normBpo{v_{\bar{\la}}}{1+\f2p} + \normBp{v}{-1+\f2p}\normBpo{v_{\bar{\la}}}{1+\f2p}
+\|v_{\bar{\la}}\|_{L^1_{t,f_1}(B^{-1+\f2p}_{p,1})}\\
& \qquad+ \f{1}{\ep}
\|v_{\bar{\la}}\|_{L^1_{t,f_2}(B^{-1+\f2p}_{p,1})} +\bigl(\mu
\|a\|_{L^\infty_t(\cM(B^{-1+\f2p}_{p,1}))}+\|\wt{\mu}(a)-\mu \|_{
L^\infty_t(\cM(B^{\f2p}_{p,1}))}\bigr)\int_0^tf_1(\tau)\,d\tau
\Bigr\}.
\end{split}
\eeq Choosing $\e, \la_1$ and $\la_2$ in \eqref{vesta} so that $C\ep
= \f{\bar{c}\mu}{4}$, $\la_1 =2C$, $\la_2=\f{8C^2}{\bar{c}\mu}$ and
$c_1\leq\f{\bar{c}}{4C},$ we obtain \beq\label{totalest1}
\begin{split}
&\normBp{v_{\bar{\la}}}{-1+\f2p}+\f{\la_1}{2}\|v_{\bar{\la}}\|_{L^1_{t,f_1}(B^{-1+\f2p}_{p,1})}
+\f{\la_2}{2}\|v_{\bar{\la}}\|_{L^1_{t,f_2}(B^{-1+\f2p}_{p,1})}+\f{\bar{c}\mu}{2}\normBpo{v_{\bar{\la}}}{1+\f2p}\\
&\leq
C_1\Bigl\{\normBp{v}{-1+\f2p}\normBpo{v_{\bar{\la}}}{1+\f2p}+\bigl(\mu
\|a\|_{L^\infty_t(\cM(B^{-1+\f2p}_{p,1}))}\\
&\qquad+\|\wt{\mu}(a)-\mu \|_{
L^\infty_t(\cM(B^{\f2p}_{p,1}))}\bigr)\int_0^tf_1(\tau)\,d\tau
\Bigr\} \quad\mbox{for}\quad t\leq\bar{T}.
\end{split}
\eeq Now let $c_2$ be a small enough positive constant, which will
be determined later on. We define $\Upsilon$ by
\beq\label{maxitimea}
\begin{split}
\Upsilon \overset{def}{=} sup \bigl\{ t<T^* :&\
\normBp{v}{-1+\f2p}+\|\wt{\mu}(a)-\mu\|_{L^\infty_t(
\cM(B^{\f2p}_{p,1}))}\\
&+\mu(\|a\|_{L^\infty_t(
\cM(B^{-1+\f2p}_{p,1}))}+\normBpo{v}{1+\f2p}) \leq c_2\mu\bigr\}.
\end{split}\eeq   \eqref{assuma} and \eqref{maxitimea} implies that $
\Upsilon\leq\bar{T}$ if we take $c_2\leq c_1$. We shall prove that
$\Upsilon= \infty$ under the assumption \eqref{thmassumea}.
Otherwise, taking $c_2\leq \f{\bar{c}}{8C_1}$, we deduce form
\eqref{totalest1} that \beq\label{5.16a}
\begin{split}
&\normBp{v_{\bar{\la}}}{-1+\f2p}+\f{\bar{c}\mu}4\normBpo{v_{\bar{\la}}}{1+\f2p}\\
&\quad\leq
 C(\mu\|a\|_{L^\infty_t(
\cM(B^{-1+\f2p}_{p,1}))}+\|\wt{\mu}(a)-\mu\|_{L^\infty_t(
\cM(B^{\f2p}_{p,1}))})\int_0^t f_1(\tau)d\tau\quad\mbox{for}\ \
t\leq \Upsilon.
\end{split}
\eeq

On the other hand, notice from \eqref{veq} that both $a$ and
$\wt{\mu}(a)-\mu$ satisfy \eqref{transeq1} so that applying
Proposition \ref{aprop1} gives rise to \beq\label{fest1}
\|a\|_{L^\infty_t(\cM(B^{-1+\f2p}_{p,1}))}\leq
C\|a_0\|_{\cM(B^{-1+\f2p}_{p,1})}\exp\Bigl\{C(\normBpo{v}{1+\f2p}+\normBpo{w}{1+\f2p})\Bigr\},
\eeq and \beq\label{fest2} \|\wt{\mu}(a)-\mu\|_{L^\infty_t(
\cM(B^{\f2p}_{p,1}))}\leq
C\|\wt{\mu}(a_0)-\mu\|_{\cM(B^{\f2p}_{p,1})}\exp\Bigl\{C(\normBpo{v}{1+\f2p}+\normBpo{w}{1+\f2p})\Bigr\}.
\eeq Then we get by summing up \eqref{5.16a} with
\eqref{fest1}$\times\mu$ and \eqref{fest2} that
 \beno
\begin{split}
&\normBp{v_{\bar{\la}}}{-1+\f2p}+
\|\wt{\mu}(a)-\mu\|_{L^\infty_t(\cM(B^{\f2p}_{p,1}))}+
\f{\mu}{4}(\|a\|_{L^\infty_t(
\cM(B^{-1+\f2p}_{p,1}))}+\normBpo{v_{\bar{\la}}}{1+\f2p})\\
&\quad\leq
C(\mu\|a_0\|_{\cM(B^{-1+\f2p}_{p,1})}+\|\wt{\mu}(a_0)-\mu\|_{\cM(B^{\f2p}_{p,1})})\exp\bigl\{C\normBpo{w}{1+\f2p}\bigr\}
\bigl(\int_0^t f_1(\tau)d\tau+1\bigr),
\end{split}
\eeno for $t\leq \Upsilon$. This together with \eqref{prop2.3} gives
rise to \beno
\begin{split}
&\normBp{v}{-1+\f2p}+ \|\wt{\mu}(a)-\mu\|_{L^\infty_t(
\cM(B^{\f2p}_{p,1}))}+ \f{\mu}{4}(\|a\|_{L^\infty_t(
\cM(B^{-1+\f2p}_{p,1}))}+\normBpo{v}{1+\f2p}) \\
&\leq
C(\mu\|a_0\|_{\cM(B^{-1+\f2p}_{p,1})}+\|\wt{\mu}(a_0)-\mu\|_{\cM(B^{\f2p}_{p,1})})\\
&\qquad\times\Bigl(\int_0^t
\bigl(\normbp{w(\tau)}{1+\f2p}+\f{1}{\mu}\normbp{\na
p(\tau)}{-1+\f2p}\bigr)d\tau+1\Bigr)\exp\bigl\{C\normBpo{w}{1+\f2p}\bigr\}\\
&\qquad\times\exp\Bigl\{4C\int_0^t\bigl(\normbp{w(\tau)}{1+\f2p}+\f{1}{\mu}\normbp{\na
p(\tau)}{-1+\f2p}
+\f{1}{\mu}\normbp{w(\tau)}{\f2p}^2\bigr)\,d\tau\Bigr\}\\
&\leq
C(\mu\|a_0\|_{\cM(B^{-1+\f2p}_{p,1})}+\|\wt{\mu}(a_0)-\mu\|_{\cM(B^{\f2p}_{p,1})})\\
&\qquad\times\exp\Bigl\{4C\int_0^t\bigl(\normbp{w(\tau)}{1+\f2p}+\f{1}{\mu}\normbp{\na
p(\tau)}{-1+\f2p}
+\f{1}{\mu}\normbp{w(\tau)}{\f2p}^2\bigr)\,d\tau\Bigr\},
\end{split}
\eeno from which and \eqref{west}, we infer \beq\label{lasta}
\begin{split}
&\normBp{v}{-1+\f2p}+ \|\wt{\mu}(a)-\mu\|_{L^\infty_t(
\cM(B^{\f2p}_{p,1}))}+ \f{\mu}{4}(\|a\|_{L^\infty_t(
\cM(B^{-1+\f2p}_{p,1}))}+\normBpo{v}{1+\f2p}) \\  &\leq
C(\mu\|a_0\|_{\cM(B^{-1+\f2p}_{p,1})}+\|\wt{\mu}(a_0)-\mu\|_{\cM(B^{\f2p}_{p,1})})\exp\Bigl\{\bar{C}(1+\mu^2)\exp\bigl(\f{\bar{C}}{\mu^2}
\normbp{u_0}{-1+\f2p}^2\bigr) \Bigr\}
\end{split}
 \eeq for $t\leq
\Upsilon$ and some positive constants $\bar{C}$ which depends on
$\bar{c}$ and $c_2$. If we take $C_0$ large enough and $c_0$
sufficiently small in \eqref{thmassumea}, there holds \beno
\normBp{v}{-1+\f2p}+ \|\wt{\mu}(a)-\mu\|_{L^\infty_t(
\cM(B^{\f2p}_{p,1}))}+ \mu(\|a\|_{L^\infty_t(
\cM(B^{-1+\f2p}_{p,1}))}+\normBpo{v}{1+\f2p}) \leq \f{c_2}{2}\mu
\eeno for $t\leq \Upsilon$, which contradicts with
\eqref{maxitimea}. Whence we conclude that $\Upsilon =T^\ast=
\infty$. This completes the proof of  Theorem \ref{mainthm1}\ef
\medskip

\noindent {\bf Acknowledgments.} We would like to thank the
anonymous referees for the valuable suggestions concerning the
structure of this paper and profitable suggestions. Part of this
work was done when M. Paicu was visiting Morningside Center of the
Chinese Academy of Sciences in the Fall of 2011. We appreciate the
hospitality and the financial support from MCM. P. Zhang is
partially supported by NSF of China under Grant 10421101 and
10931007, the one hundred talents' plan from Chinese Academy of
Sciences under Grant GJHZ200829 and the National Center for
Mathematics and Interdisciplinary Sciences, CAS.

\end{document}